\documentclass[10pt]{article}

\usepackage[top=1in, bottom=1in, left=1in, right=1in]{geometry}

\usepackage{setspace}
\usepackage{algorithm}

\usepackage{algpseudocode}

\usepackage{amssymb}

\usepackage{amsthm}

\usepackage{amsmath}

\usepackage{adjustbox}

\usepackage{bm}

\usepackage{bbm}

\usepackage{natbib}

\usepackage{graphicx}

\usepackage{booktabs}

\usepackage{blkarray}

\usepackage{titlesec}  

\usepackage{enumerate}

\usepackage{authblk}
\usepackage[labelfont=bf, textfont=it]{caption}

\DeclareMathOperator*{\argmin}{argmin}   %
\newcommand{\mat}[1]{\bm{#1}}

\newcommand{\vect}[1]{\bm{#1}}

\newcommand{\T}{\mathsf{T}}

\newtheorem{lemma}{Lemma}
\newtheorem{theorem}{Theorem}

\theoremstyle{definition}

\newtheorem{definition}{Definition}
\newtheorem{remark}{Remark}
\newtheorem{assumption}{Assumption}

\title{On randomized sketching algorithms and the Tracy-Widom law}

\author[*]{Daniel Ahfock}
\author[ ]{William J. Astle}
\author[ ]{Sylvia Richardson}
\affil[ ]{MRC Biostatistics Unit, University of Cambridge}
\affil[*]{\texttt{d.ahfock@uq.edu.au}}
\date{}                     
\begin{document}
\maketitle

\begin{abstract}
There is an increasing body of work exploring the integration of random projection into algorithms for numerical linear algebra. The primary motivation is to reduce the overall computational cost of processing large datasets. A suitably chosen random projection can be used to embed the original dataset in a lower-dimensional space such that key properties of the original dataset are retained. These algorithms are often referred to as sketching algorithms, as the projected dataset can be used as a compressed representation of the full dataset. We show that random matrix theory, in particular the Tracy-Widom law, is useful for describing the operating characteristics of sketching algorithms in the tall-data regime when $n \gg d$. Asymptotic large sample results are of particular interest as this is the regime where sketching is most useful for data compression. In particular, we develop asymptotic approximations for the success rate in generating random subspace embeddings and the convergence probability of iterative sketching algorithms. We test a number of sketching algorithms on real large high-dimensional datasets and find that the asymptotic expressions give accurate predictions of the empirical performance.
\end{abstract}

\section{Introduction}
Sketching is a probabilistic data compression technique that makes use of random projection \citep{cormode_sketch_2011, mahoney_randomized_2011, woodruff_sketching_2014}. Suppose interest lies in a $n \times d$ dataset $\mat{A}$. When $n$ and or $d$ are large, typical data analysis tasks will involve a heavy numerical computing load. This computational burden can be a practical obstacle for statistical learning with Big Data. When the sample size $n$ is the computational bottleneck, sketching algorithms use a linear random projection to create a smaller sketched dataset of size $k \times d$, where $k \ll n$. The random projection can be represented as a $k \times n$ random matrix $\mat{S}$, and the sketched dataset $\widetilde{\mat{A}}$ is generated through the linear embedding $\widetilde{\mat{A}}=\mat{S}\mat{A}$. The smaller sketched dataset $\widetilde{\mat{A}}$ is used as a surrogate for the full dataset $\mat{A}$ within numerical routines. Through a judicious choice of the distribution on the random sketching matrix $\mat{S}$, it is often possible to bound the error that is introduced stochastically into calculations given the use of the randomized approximation $\widetilde{\mat{A}}$ in place of $\mat{A}$

The selected distribution of the random sketching matrix $\mat{S}$ can be divided into two categories, data-oblivious sketches, where the distribution is not a function of the source data $\mat{A}$, and data-aware sketches, where the distribution is a function of $\mat{A}$. The majority of data-aware sketches perform weighted sampling with replacement, and are closely connected to finite population survey sampling methods \citep{ma_statistical_2015, quiroz_2018_subsampling}. The analysis of data-oblivious sketches requires different methods to data-aware sketches, as there are no clear ties to finite-population subsampling. In general, data-oblivious sketches generate a dataset of $k$ pseudo-observations, where each instance in the compressed representation $\widetilde{\mat{A}}$ has no exact counterpart in the original source dataset $\mat{A}$. 

Three important data-oblivious sketches are the Gaussian sketch, the Hadamard sketch and the Clarkson-Woodruff sketch. The Gaussian sketch is the simplest of these, where each element in the $k \times n$ matrix $\mat{S}$ is an independent sample from a $N(0, 1/k)$ distribution. The Hadamard sketch uses structured elements for fast matrix multiplication, and the Clarkson-Woodruff uses sparsity in $\mat{S}$ for efficient computation of the sketched dataset. The comparative performance between distributions on $\mat{S}$ is of interest, as there is a trade-off between the computational cost of calculating $\widetilde{\mat{A}}$ and the fidelity of the approximation $\widetilde{\mat{A}}$ with respect to original $\mat{A}$ when choosing the type of sketch. Our results help to establish guidelines for selecting the sketching distribution.

Sketching algorithms are typically framed using stochastic $(\delta, \epsilon)$  error bounds, where the algorithm is shown to attain $(1 \pm \epsilon)$ accuracy with probability at least $1-\delta$ \citep{woodruff_sketching_2014}.  These notions are made more precise in Section \ref{sec:sketching}. Existing bounds are typically developed from a worst-case non-asymptotic viewpoint \citep{mahoney_randomized_2011, woodruff_sketching_2014, tropp_improved_2011}. We take a different approach, and use random matrix theory to develop asymptotic approximations to the success probability given the sketching distortion factor $\epsilon$. 

Our main result is an asymptotic expression for the probability that a Gaussian based sketching algorithm satisfies general $(1 \pm \epsilon)$ probabilistic error bounds in terms of the Tracy-Widom law (Theorem \ref{thm:gaussian_embedding_tw_limit}), which describes the distribution of the extreme eigenvalues of large random matrices \citep{tracy_level_1994, johnstone_distribution_2001}. We then identify regularity conditions where other data-oblivious projections are expected to demonstrate the same limiting behavior (Theorem \ref{thm:data_oblivious_asymptotic_embedding}). If the motivation for using a sketching algorithm is data compression due to large $n$, the asymptotic approximations are of particular interest as they become more accurate as the computational benefits afforded by the use of a sketching algorithm increase in tandem. Empirical work has found that the quality of results can be consistent across the choice of random projections \citep{venkata_johnson_2011, le_2013_fastfood, dahiya_2018_empirical}, and our results shed some light on this issue. An application is to determine the convergence probability when sketching is used in iterative least-squares optimisation.  We test the asymptotic theory and find good agreement on datasets with large sample sizes $n \gg d$. Our theoretical and empirical results show that random matrix theory has an important role in the analysis of data-oblivious sketching algorithms for data compression.

\section{Sketching}
\label{sec:sketching}
\subsection{Data-oblivious sketches}
\label{subsec:sketching}
As mentioned, a key component in a sketching algorithm is the distribution on $\mat{S}$. Four important random linear maps are:

\begin{itemize}
    \item The uniform sketch implements subsampling uniformly with replacement followed by a rescaling step. The Uniform projection can be represented as  $\mat{S}=\sqrt{n/k}\Phi$.  The random matrix $\Phi$ subsamples $k$  rows of $\mat{A}$ with replacement. Element $\Phi_{r,i}=1$ if observation $i$ in the source dataset is selected in the $r$th subsampling round $(r=1, \ldots, k; \ i=1\ldots, n)$. The uniform sketch can be implemented in $O(k)$ time.
    \item A Gaussian sketch is formed by  independently sampling each element of $\mat{S}$ from a $N(0, 1/k)$ distribution. Computation of the sketched data is $O(ndk)$. 
    \item The Hadamard sketch is a structured random matrix \citep{ailon_fast_2009}. The sketching matrix is formed as $\mat{S} = \Phi\mat{H}\mat{D}/\sqrt{k}$, where $\Phi$ is a $k \times n$ matrix and $\mat{H}$ and $\mat{D}$ are both $n \times n$ matrices. The fixed matrix $\mat{H}$ is a Hadamard matrix of order $n$. A Hadamard matrix is a square matrix with elements that are either $+1$ or $-1$ and orthogonal rows. Hadamard matrices do not exist for all integers $n$, the source dataset can be padded with zeroes so that a conformable Hadamard matrix is available. The random matrix $\mat{D}$ is a diagonal matrix where each of the $n$ diagonal entries is an independent Rademacher random variable. The random matrix $\Phi$ subsamples $k$  rows of $\mat{H}$ with replacement. The structure of the Hadamard sketch allows for fast matrix multiplication, reducing calculation of the sketched dataset to $O(nd \log k)$ operations. 

    \item The Clarkson-Woodruff sketch is a sparse random matrix \citep{clarkson_low_2013}. The  projection can be represented as the product of two independent random matrices, $\mat{S} = \mat{\Gamma}\mat{D}$, where $\mat{\Gamma}$ is a random $k \times n$ matrix and $\mat{D}$ is a random $n \times n$ matrix. The matrix $\mat{\Gamma}$ is initialized as a matrix of zeros. In each column, independently, one entry is selected and set to $+1$. The matrix $\mat{D}$ is a diagonal matrix where each of the $n$ diagonal entries is an independent Rademacher random variable. This results in a sparse $\mat{S}$, where there is only one nonzero entry per column. The sparsity of the Clarkson-Woodruff sketch speeds up matrix multiplication, dropping the complexity of  generating the sketched dataset to $O(nd)$.
\end{itemize}

The Gaussian sketch was central to early work on sketching algorithms \citep{sarlos_improved_2006}. The drawback of the Gaussian sketch is that computation of the sketched data is quite demanding, taking $O(ndk)$ operations. As such, there has been work on designing more computationally efficient random projections.

Sketch quality is commonly measured using  $\epsilon$-subspace embeddings  (\citet[Chapter 2]{woodruff_sketching_2014}, \cite{meng_low-distortion_2013}, \cite{yang_implementing_2015}). These are defined below. 

\begin{definition}{\emph{$\epsilon$-subspace embedding}}
\label{defn:epsilon_subspace_embedding}\newline
For a given $n \times d$ matrix $\mat{A}$, we call a $k \times n $ matrix $\mat{S}$ an $\epsilon$-subspace embedding for $\mat{A}$, if for all vectors $\vect{z} \in \mathbb{R}^{d}$
\begin{align*}
(1-\epsilon)|| \mat{A}\vect{z}||_{2}^{2} \le ||\mat{S} \mat{A}\vect{z}||_{2}^{2} \le (1+\epsilon)|| \mat{A}\vect{z}||_{2}^{2}.
\end{align*}
\end{definition}
An $\epsilon$-subspace preserves the linear structure of the original dataset up to a multiplicative $(1 \pm \epsilon)$ factor. Broadly speaking, the covariance matrix of the sketched dataset $\widetilde{\mat{A}}=\mat{S}\mat{A}$ is similar to the covariance matrix of the source dataset $\mat{A}$ if $\epsilon$ is small. Mathematical arguments show that the sketched dataset is a good surrogate for many linear statistical methods if the sketching matrix $\mat{S}$ is an $\epsilon$-subspace embedding for the original dataset, with $\epsilon$ sufficiently small  \citep{woodruff_sketching_2014}. Suitable ranges for $\epsilon$ depend on the task of interest and structural properties of the source dataset \citep{mahoney_structural_2016}. 

The Gaussian, Hadamard and Clarkson-Woodruff projections are popular data-oblivious projections as it is possible to argue that they produce $\epsilon$-subspace embeddings with high probability for an arbitrary data matrix $\mat{A}$.  It is considerably more difficult to establish universal worst case bounds for the uniform projection \citep{drineas_sampling_2006, ma_statistical_2015}.  We include the uniform projection in our discussion as it is a useful baseline. 

\begin{table}[h]
	\centering
\begin{tabular}{@{} l  l l l @{}}
		\toprule
			Sketch & Sketching time  & Required sketch size $k$ \\
			\midrule
			Gaussian  & $O(ndk) $  & $O((d+\log(1/\delta))/\epsilon^2) $\\
			Hadamard & $O(nd \log k)$ & $O((\sqrt{d}+\sqrt{\log n})^2(\log (d/\delta))/ \epsilon^2)$\\
			Clarkson-Woodruff  & $O(nd)$& $O(d^2/(\delta\epsilon^2))$ \\
			Uniform & $O(k)$ &  $-$ & \\
			\bottomrule
	\end{tabular}
	\caption[Properties of different data-oblivious random projections.]{\label{tab:run_time}Properties of different data-oblivious random projections (see \citet{woodruff_sketching_2014} and the references therein). The third column refers to the necessary sketch size $k$ to obtain an $\epsilon$-subspace embedding for an arbitrary $n \times d$ source dataset with at least probability $(1-\delta)$.}
\end{table}

\subsection{Sketching algorithms}
Sketching algorithms have been proposed for key linear statistical methods such as low rank matrix approximation, principal components analysis, linear discriminant analysis and ordinary least squares regression  \citep{mahoney_randomized_2011, woodruff_sketching_2014, erichson_randomized_2016, falcone_2021_matrix}. Sketching has also been investigated for Bayesian posterior approximation \citep{bardenet_note_2015, geppert_random_2017}. A common thread throughout these works is the reliance on the generation of an $\epsilon$-subspace embedding. In general, $\epsilon$ serves an approximation tolerance parameter, with smaller $\epsilon$ guaranteeing higher fidelity to exact calculation with respect to some  divergence measure. 

An example application of sketching is ordinary least squares regression \citep{sarlos_improved_2006}. The sketched responses and predictors are defined as $\widetilde{\vect{y}}=\mat{S}\vect{y}, \widetilde{\mat{X}}=\mat{S}\mat{X}$. Let $\vect{\vect{\beta}}_{F} = \argmin_{\vect{\beta}}\lVert \vect{y}-\mat{X}\vect{\beta} \rVert_{2}^{2}, \vect{\beta}_{S} = \argmin_{\vect{\beta}}\lVert \widetilde{\vect{y}}-\widetilde{\mat{X}}\vect{\beta} \rVert_{2}^{2}$, and $RSS_{F}=\lVert \vect{y}-\mat{X}\vect{\beta}_{F}\rVert_{2}^{2}$. It is possible to establish the concrete bounds, that if ${\mat{S}}$ is an $\epsilon$-subspace embedding for $\mat{A}=(\vect{y}, {\mat{X}})$  \citep{sarlos_improved_2006}, then
\begin{align*}
\lVert \vect{\beta}_{S}- \vect{\beta}_{F}\rVert_{2}^{2} &\le  \dfrac{\epsilon^2}{\sigma_{\text{min}}^2(\mat{X})}RSS_{F}, 
\end{align*}
where $\sigma_{\text{min}}(\mat{X})$ represents the smallest singular value of the design matrix $\mat{X}$. If $\epsilon$ is very small, then $\vect{\beta}_{S}$ is a good approximation to $\vect{\beta}_{F}$.

Given the central role of $\epsilon$-subspace embeddings (Definition \ref{defn:epsilon_subspace_embedding}), the success probability,
\begin{align}
\Pr (\mat{S} \text{ is an $\epsilon$-subspace embedding for $\mat{A}$})
\end{align}
is thus an important descriptive measure of the uncertainty attached to the randomized algorithm. The probability statement is over the random sketching matrix $\mat{S}$ with the dataset $\mat{A}$ treated as fixed. The embedding probability is difficult to characterize precisely using existing theory \citep{venkata_johnson_2011}.  The bounds in Table \ref{tab:run_time} only give qualitative guidance about the embedding probability. Users will benefit from more prescriptive results in order to choose the sketch size $k$, and the type of sketch for applications \citep{grellmann_random_2016, geppert_random_2017, ahfock_statistical_2020, falcone_2021_matrix}.

Another use for sketching is in iterative solvers for ordinary least squares regression. A sketch  $\widetilde{\mat{X}} = \mat{S}\mat{X}$ can be used to generate a random preconditioner, $(\widetilde{\mat{X}}^{\T}\widetilde{\mat{X}})^{-1}$,  that is then applied to the normal equations $\mat{X}^{\T}\mat{X}\vect{\beta}=\mat{X}^{\T}\vect{y}$. Given some initial value $\vect{\beta}^{(0)}$, the iteration is defined as
\begin{align}
    \vect{\beta}^{(t+1)} &= \vect{\beta}^{(t)} + (\widetilde{\mat{X}}^{\T}\widetilde{\mat{X}})^{-1}\mat{X}^{\T}(\vect{y}-\mat{X}\vect{\beta}^{(t)}). \label{eq:basic_iteration}
\end{align}
If $\widetilde{\mat{X}}^{\T}\widetilde{\mat{X}}=\mat{X}^{\T}\mat{X}$ the iteration will converge in a single step. The degree of noise in the preconditioner will be influenced by the sketch size $k$. A sufficient condition for convergence of the iteration \eqref{eq:basic_iteration} is that $\mat{S}$ is an $\epsilon$-subspace embedding for $\mat{X}$ with $\epsilon < 0.5$ \citep{pilanci_iterative_2016}. As is typical with randomized algorithms, we accept some failure probability in order to relax the computational demands.   It is of interest to develop expressions for the failure probability of the algorithm as a function of the sketch size $k$, as this can give useful guidelines in practice. It is possible to establish worst case bounds using the results in Table \ref{tab:run_time}, however we will aim to give a point estimate of the probability. Although it is possible to improve on the iteration \eqref{eq:basic_iteration} using acceleration methods \citep{meng_2014_lsrn, dahiya_2018_empirical, lacotte_2020_limiting}, we focus on the basic iteration to introduce our asymptotic techniques. 

\subsection{Operating characteristics}
 Let the singular value decomposition of the source dataset be given by $\mat{A}=\mat{U}\mat{D}\mat{V}^{\T}$.  Let $\sigma_{\text{min}}(\mat{M})$ and $\sigma_{\text{max}}(\mat{M})$ denote the minimum and maximum singular values respectively, of a matrix $\mat{M}$. Likewise, let $\lambda_{\text{min}}(\mat{M})$ and $\lambda_{\text{max}}(\mat{M})$ denote the minimum and maximum eigenvalues of a matrix $\mat{M}$. It is possible to show
\begin{align}
\Pr (\mat{S} \text{ is an $\epsilon$-subspace embedding for $\mat{A}$}) &=  \Pr(\sigma_{\text{max}}(\mat{I}_{d}-\mat{U}^{\T}\mat{S}^{\T}\mat{S}\mat{U}) \le \epsilon), \label{eq:embedding_probability}
\end{align} 
where $\mat{U}$ is the $n \times d$ matrix of left singular vectors of the source data matrix $\mat{A}$ \citep{woodruff_sketching_2014}. Now as  
\begin{align}
\sigma_{\text{max}}(\mat{I}_{d}-\mat{U}^{\T}\mat{S}^{\T}\mat{S}\mat{U}) &= \text{max}(\lvert 1-\lambda_{\text{min}}(\mat{U}^{\T}\mat{S}^{\T}\mat{S}\mat{U}) \rvert, \lvert 1-\lambda_{\text{max}}(\mat{U}^{\T}\mat{S}^{\T}\mat{S}\mat{U}) \rvert) , \label{eq:max_identity} 
\end{align}
the extreme eigenvalues of $\mat{U}^{\T}\mat{S}^{\T}\mat{S}\mat{U}$ are the critical factor in generating $\epsilon$-subspace embeddings. The convergence behavior of the basic iteration \eqref{eq:basic_iteration} is also tied to the eigenvalues of $\mat{U}^{\T}\mat{S}^{\T}\mat{S}\mat{U}$ where $\mat{A}=\mat{X}$. Providing that $(\widetilde{\mat{X}}^{\T}\widetilde{\mat{X}})$ is of rank $d$, the maximum eigenvalue satisfies
 \begin{align*}
    \lambda_{\text{max}}((\widetilde{\mat{X}}^{\T}\widetilde{\mat{X}})^{-1}\mat{X}^{\T}\mat{X}) &=  \lambda_{\text{max}}((\mat{U}^{\T}\mat{S}^{\T}\mat{S}\mat{U})^{-1}).
 \end{align*}
 From standard results on iterative solvers \citep{hageman_2012_applied}, a necessary and sufficient condition for the iteration to converge is $
    \underset{t \to \infty}{\lim}\lVert\vect{\beta}_{F} - \vect{\beta}^{(t)} \rVert_{2} = 0$ if and only if $ \lambda_{\text{max}}((\widetilde{\mat{X}}^{\T}\widetilde{\mat{X}})^{-1}\mat{X}^{\T}\mat{X}) < 2$. The probability of convergence can then be expressed as
\begin{align}
   \Pr\left(\underset{t \to \infty}{\lim}\lVert\vect{\beta}_{F} - \vect{\beta}^{(t)} \rVert_{2} = 0\right) &=  \Pr(\lambda_{\text{min}}(\mat{U}^{\T}\mat{S}^{\T}\mat{S}\mat{U}) > 0.5). \label{eq:convergence_probability}
\end{align}
Most existing results on the probabilities \eqref{eq:embedding_probability} and \eqref{eq:convergence_probability} are finite sample lower bounds \citep{tropp_improved_2011, nelson_osnap_2013, meng_randomized_2014}. Worst case bounds can be conservative in practice, and there is value in developing other methods to characterize the performance of randomized algorithms \citep{ halko_finding_2011,raskutti_statistical_2014, lopes_2018_error, dobriban_2018_new}. The embedding probability  \eqref{eq:embedding_probability}  and the convergence probability \eqref{eq:convergence_probability} are related to the extreme eigenvalues of $\mat{U}^{\T}\mat{S}^{\T}\mat{S}\mat{U}$. In Section \ref{sec:gaussian} we study this distribution for the Gaussian sketch and develop a Tracy-Widom approximation. The approximation is then extended to the Clarkson-Woodruff and Hadamard sketches in Section \ref{sec:data_oblivious}. 

\section{Gaussian sketch}
\label{sec:gaussian}
\subsection{Exact representations}
\citet[Section 2.3]{meng_randomized_2014} notes that when using a Gaussian sketch, it is instructive to  consider {directly} the distribution of the random variable $
\sigma_{\text{max}}(\mat{I}_{d}-\mat{U}^{\T}\mat{S}^{\T}\mat{S}\mat{U})$ to study the embedding probability \eqref{eq:embedding_probability}. Consider an arbitrary $n \times d$ data matrix $\mat{A}$. As $\mat{S}$ is a matrix of independent Gaussians with mean zero and variance $1/k$,  it is possible to show that
\begin{align*}
\mat{U}^{\T}\mat{S}^{\T}\mat{S}\mat{U} &\sim  \text{Wishart}\left(k, \mat{I}_{d}/k\right) ,
\end{align*}
as $\mat{U}^{\T}\mat{U}=\mat{I}_{d}$. The key term $\mat{U}^{\T}\mat{S}^{\T}\mat{S}\mat{U}$ is in some sense a pivotal quantity, as its distribution is invariant to the actual values of the data matrix $\mat{A}$. When using a Gaussian sketch, the probability of obtaining an $\epsilon$-subspace embedding has no dependence on the number of original observations $n$, or on the values in the data matrix  $\mat{A}$. This is a useful property for a data-oblivious sketch, as it is possible to develop universal performance guarantees that will hold for any possible source dataset. This invariance property is also noted in \citet{meng_randomized_2014}, although the derivation is different.

Let us define the random matrix $\mat{W} \sim \text{Wishart}(k, \mat{I}_{d}/k)$. The success probability of interest can then be expressed in terms of the extreme eigenvalues of the Wishart distribution
The embedding probability of interest has the representation:
\begin{align}
\Pr (\mat{S} \text{ is an $\epsilon$-subspace embedding for $\mat{A}$}) &=
 \Pr\left( | 1- \lambda_{\text{min}}(\mat{W})| \le \epsilon,   | 1- \lambda_{\text{max}}(\mat{W})| \le \epsilon  \right). \label{eq:wishart_embedding_eigenvalues}
\end{align}
where we have made use of the expression for the maximum singular value \eqref{eq:max_identity}.

It is difficult to obtain a mathematically tractable expression for the embedding probability as it involves the joint distribution of the extreme eigenvalues \citep{chiani_2017_probability}. \citeauthor{meng_randomized_2014} forms a lower bound on the probability \eqref{eq:wishart_embedding_eigenvalues} using concentration results on the eigenvalues of the Wishart distribution. 

The convergence probability \eqref{eq:convergence_probability}, can also be related to the eigenvalues of the Wishart distribution. Assuming $k \ge d$, the matrix $\widetilde{\mat{X}}^{\T}\widetilde{\mat{X}}$ has full rank with probability one. As such, using the same pivotal quantity $\mat{U}^{\T}\mat{S}^{\T}\mat{S}\mat{U}$ as before,
\begin{align}
   \Pr\left(\underset{t \to \infty}{\lim}\lVert\vect{\beta}_{F} - \vect{\beta}^{(t)} \rVert_{2} = 0\right)&= \Pr(\lambda_{\text{min}}(\mat{W}) > 0.5), \label{eq:conv_prob}
\end{align}
where $\mat{W}\sim  \text{Wishart}(k, \mat{I}_{d}/k)$. The convergence probability \eqref{eq:conv_prob} has no dependence on the specific response vector $\vect{y}$ or design matrix $\mat{X}$ under consideration. Problem invariance is a  highly desirable property for a randomized iterative solver \citep{roosta_subsampled_2016, lacotte_2020_limiting}. Both the embedding probability and the convergence probability are related to the extreme eigenvalues of the Wishart distribution. The extreme eigenvalues of Wishart random matrices are a well studied topic in random matrix theory \citep{edelman_eigenvalues_1988}, and we can make use of existing results to analyse the operating characteristics of sketching algorithms. In the following section we develop approximations to the embedding probability and the convergence probability in the  asymptotic regime: 
\begin{align}
    n,d,k \to \infty, \quad n > k,  \quad d/k \to \alpha \in (0, 1]. \label{eq:big_data_regime}
\end{align}
The limit is asymptotic in $n$, $d$ and $k$, with the constraint that the number of variables to sketch size tends to a constant $\alpha$. This can be interpreted as a type of Big Data asymptotic, where we consider tall and wide datasets through the limit in $n$ and $d$, and increasing sketch sizes $k$ to cope with the expanding number of variables $d$. Although there is no explicit dependence on $n$ for the finite sample expressions \eqref{eq:embedding_probability} and \eqref{eq:conv_prob} for the Gaussian sketch, the asymptotic limit in $n$ is still used to emphasize that we are taking limits in the tall-data setting. 

\citet{dobriban_2018_new} analyse the mean squared error of single-pass sketching algorithms for linear regression in this asymptotic framework under the assumption of a generative model.  Our analysis is different as we are concerned with the embedding and convergence  probabilities (\eqref{eq:embedding_probability} and \eqref{eq:convergence_probability}), rather than the accuracy of population parameter estimates. In independent work, \citet{lacotte_2020_limiting} study the limiting empirical spectral distribution of Hadamard sketch in the asymptotic regime \eqref{eq:big_data_regime}. Here we are concerned with the fluctuations of the extreme eigenvalues rather than the bulk of the spectrum. 
\subsection{Random matrix theory}
Random matrix theory involves the analysis of large random matrices \citep{bai_spectral_2010}. The Tracy-Widom law is an important result in the study of the extreme eigenvalue statistics \citep{tracy_level_1994}. \citet{johnstone_distribution_2001} showed that Tracy-Widom law gives the asymptotic distribution of the maximum eigenvalue of a $\text{Wishart}(k, \mat{I}_{d}/k)$ matrix after appropriate centering and scaling. In subsequent work \cite{ma_accuracy_2012} showed that the rate of convergence could be improved from ${O}(d^{-1/3})$ to ${O}(d^{-2/3})$ by using different centering and scaling constants than in \cite{johnstone_distribution_2001}. We build from the convergence result given by \citeauthor{ma_accuracy_2012}. 

The \texttt{R} package \texttt{RMTstat} contains a number of functions for working with the Tracy-Widom distribution \citep{johnstone_2014_rmtstat}. The main application of the Tracy-Widom law to statistical inference has been its use in hypothesis testing in high-dimensional statistical models \citep{johnstone_high_2006, bai_spectral_2010}. To the best of our knowledge, the connection to sketching algorithms has not been explored in great depth. The Tracy-Widom law can be used to approximate the embedding probability \eqref{eq:embedding_probability}.

\begin{theorem}
\label{thm:gaussian_embedding_tw_limit}
Suppose we have an arbitrary $n \times d$ data matrix $\mat{A}$ where $n >d$ and $\mat{A}$ is of rank $d$.  Furthermore assume we take a Gaussian sketch of size $k$. Consider the limit in $n, k$ and $d$, such that $d/k \to \alpha$ with $\alpha \in (0, 1]$. Define centering and scaling constants $\mu_{k,d}$ and $\sigma_{k,d}$ as 
\begin{align*}
\mu_{k,d} &= k^{-1}(\sqrt{k-1/2}+\sqrt{d-1/2})^{2}, \quad \sigma_{k,d} = \dfrac{k^{-1}(\sqrt{k-1/2}+\sqrt{d-1/2})}{\left(1/{\sqrt{k-1/2}}+1/{\sqrt{d-1/2}}\right)^{1/3}}. 
\end{align*}
Set $Z \sim F_{1}$ where $F_{1}$ is the Tracy-Widom distribution.  Let $\psi_{n,k,d}$ give the exact embedding probability and let $\widehat{\psi}_{n,k,d}$ give the asymptotic approximation to the embedding probability:
\begin{align*}
    \psi_{n,k,d} &= \Pr \left(\mat{S} \emph{ is an $\epsilon$-subspace embedding for $\mat{A}$}\right), \quad \widehat{\psi}_{n,k,d} = \Pr \left( Z \le \dfrac{\epsilon +1-\mu_{k,d}}{\sigma_{k,d}}\right).
\end{align*}
Then asymptotically in $n, d$ and $k$, for any $\epsilon>0$,
\begin{align*}
\underset{n,d,k \to \infty}{\lim} \left\lvert \psi_{n,k,d} - \widehat{\psi}_{n,k,d}  \right\rvert &=0
\end{align*}
Furthermore, for even $d$, $ \left\lvert \gamma_{n,k,d}-\widehat{\gamma}_{n,k,d} \right\rvert = {O}(d^{-2/3})$.
\end{theorem}
The proof is given in the supplementary material.

The convergence probability of the iterative algorithm \eqref{eq:convergence_probability} can also be approximated using the Tracy-Widom law. 
\begin{theorem}
\label{thm:gaussian_convergence_tw}
Suppose we have an arbitrary $n \times d$ data matrix $\mat{A}$ where $n >d$ and $\mat{A}$ is of rank $d$.  Furthermore, assume we take a Gaussian sketch of size $k$. Consider the limit in $n, k$ and $d$, such that $d/k \to \alpha$ with $\alpha \in (0, 1]$. Set
\begin{align*}
\mu_{k,d} &= (\sqrt{k-1/2}-\sqrt{d-1/2})^{2},\\
\sigma_{k,d} &= (\sqrt{k-1/2}-\sqrt{d-1/2})\left(\dfrac{1}{\sqrt{k-1/2}}-\dfrac{1}{\sqrt{d-1/2}}\right)^{1/3},
\end{align*}
and define the following centering and scaling constants $ \tau_{k,d} = \sigma_{k,d}/\mu_{k,d},  \nu_{k,d} = \log(\mu_{k,d})-\log k -\tau_{k,d}^2/8$. Set $Z \sim F_{1}$, where $F_{1}$ is the Tracy-Widom distribution. Let $\gamma_{n,k,d}$ give the exact convergence probability, and $\widehat{\gamma}_{n,k,d}$ give the asymptotic approximation to the convergence probability:
\begin{align*}
    \gamma_{n,k,d} &= \Pr \left(\underset{t \to \infty}{\lim}\lVert\vect{\beta}_{F} - \vect{\beta}^{(t)} \rVert_{2} = 0\right), \quad
    \widehat{\gamma}_{n,k,d} =  \Pr\left(Z \le \dfrac{\nu_{k,d}-\log (1/2)}{\tau_{k,d}}\right).
\end{align*}
Then for all starting values $\vect{\beta}^{(0)}$, asymptotically in $n, d$ and $k$, 
\begin{align*}
   \quad \underset{n,d,k \to \infty}{\lim} \left\lvert \gamma_{n,k,d}-\widehat{\gamma}_{n,k,d} \right\rvert = 0.
\end{align*}
Furthermore, for even $d$, $ \left\lvert \gamma_{n,k,d}-\widehat{\gamma}_{n,k,d} \right\rvert = {O}(d^{-2/3})$.
\end{theorem}
The proof is given in the supplementary material.

The  embedding probability for the Gaussian sketch can be estimated by simulating $\mat{W} \sim \text{Wishart}(k, \mat{I}_{d}/k)$ and using the empirical distribution  of the random variable $\sigma_{\text{max}}\left(\mat{I}_{d} - \mat{W}\right)$. To assess the accuracy of the approximation in Theorem \ref{thm:gaussian_embedding_tw_limit}, we generated $B=10,000$ random Wishart matrices $\mat{W}^{[1]}, \ldots, \mat{W}^{[B]}$. For each simulated matrix $\mat{W}^{[b]}$ we computed the distortion factor $\epsilon^{[b]} = \sigma_{\text{max}}(\mat{I}_{d}-\mat{W}^{[b]})$ for $b=1, \ldots, B$. The simulated distortion factors $\epsilon^{[1]}, \ldots, \epsilon^{[B]}$ were used to give a Monte Carlo estimate of the embedding probability:
\begin{align}
\widehat{\Pr}(\mat{S} \text{ is an } \epsilon\text{-subspace embedding for } \mat{A}) &=  \dfrac{1}{B}\sum_{b=1}^{B}\mathbbm{1}(\epsilon^{[b]} \le \epsilon). \label{eq:monte_carlo_embedding}
\end{align}

We used the \texttt{ARPACK} library \citep{lehoucq_arpack_1998} to compute the maximum singular values $\sigma_{\text{max}}(\mat{I}_{d}-\mat{W}^{[b]})$. The estimated embedding probabilities are displayed in Figure \ref{fig:epsilon_ecdf} for different dimensions $d$. The sketch size to variables ratio, $k/d$, was held fixed at 20. The solid red line shows the empirical probability of obtaining an $\epsilon$-subspace embedding. The dashed black line gives the Tracy-Widom approximation given in Theorem \ref{thm:gaussian_embedding_tw_limit}. The agreement is consistently good over dimensions $d$, and the range of sketch sizes $k$ that were considered. 
\begin{figure*}
    \centering
    \includegraphics[width=0.9\textwidth]{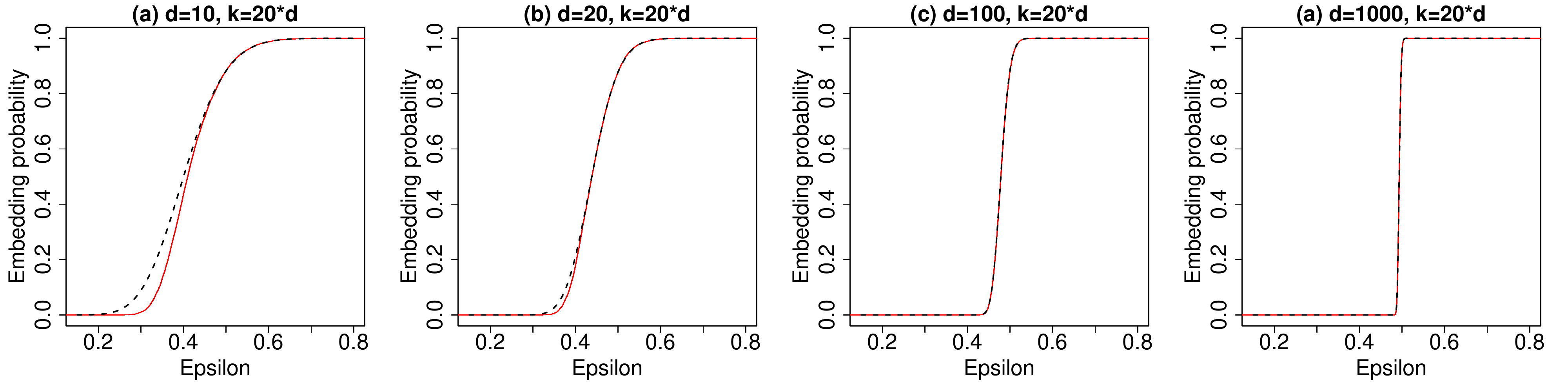}
    \caption{Accuracy of Tracy-Widom approximation for embedding probability \eqref{eq:wishart_embedding_eigenvalues} for the Gaussian sketch. The dashed black line gives the asymptotic limit, the solid red line gives the empirical probability. When $d\ge 20$ the approximation given in Theorem \ref{thm:gaussian_embedding_tw_limit} is very accurate. }
    \label{fig:epsilon_ecdf}
\end{figure*}

\section{Computationally efficient sketches}
\label{sec:data_oblivious}
\subsection{Asymptotics}
Asymptotic methods are useful to analyse data-oblivious sketches that do not admit interpretable finite sample distributions \citep{li_very_2006, ahfock_statistical_2020, lacotte_2020_limiting}. Here we describe the limiting behavior of the sketched algorithms for fixed $k$ and $d$ as the number of source observations $n$ increases. 

Under an assumption on the limiting leverage scores of the source data matrix, we can establish a limit theorem for the Hadmard and Clarkson-Woodruff sketches.  The leverage scores are an important structural property in sketching algorithms \citep{mahoney_structural_2016}. 

\begin{assumption}
\label{assum:leverage}
Define the singular value decomposition of the $n \times d$ source dataset as $\mat{A}_{(n)}=\mat{U}_{(n)}\mat{D}_{(n)}\mat{V}_{(n)}^{\T}$. Let $\vect{u}_{(n)i}^{\T}$ give the $i$th row in $\mat{U}_{(n)}$. Assume that the maximum leverage score tends to zero, that is
	\begin{align*}
	\lim_{n \to \infty} \underset{i=1, \ldots, n}{\textnormal{max}} \lVert \vect{u}_{(n)i} \rVert_{2}^{2} = 0. 
	\end{align*}
\end{assumption}

The asymptotic probability of obtaining an $\epsilon$-subspace embedding for the Hadamard and Clarkson-Woodruff sketches can be related to the Wishart distribution. 
\begin{theorem}
\label{thm:data_oblivious_asymptotic_embedding}
Consider a sequence of arbitrary $n \times d$ data matrices $\mat{A}_{(n)}$, where each data matrix is of rank $d$, and $d$ is fixed. Let $\mat{A}_{(n)}=\mat{U}_{(n)}\mat{D}_{(n)}\mat{V}_{(n)}^{\T}$ represent the singular value decomposition of $\mat{A}_{(n)}$.  Let $\mat{S}_{(n)}$ be a $k \times n$ Hadamard or Clarkson-Woodruff sketching matrix where $k$ is also fixed. Suppose that Assumption \ref{assum:leverage} is satisfied. Then as $n$ tends to infinity with $k$ and $d$ fixed, 
\begin{align*}
\underset{n \to \infty}{\lim}\Pr\left( \mat{S}_{(n)} \emph{ is an $\epsilon$-subspace embedding for $\mat{A}_{(n)}$} \right) &= \Pr\left(\sigma_{\emph{max}}(\mat{I}_{d}-\mat{W}) \le \epsilon \right),
\end{align*}
where $\mat{W} \sim \emph{Wishart}(k, \mat{I}_{d}/k)$.
\end{theorem}
The proof is given in the supplementary material. 

Theorem \ref{thm:data_oblivious_asymptotic_embedding} states the the embedding probability for the Hadamard and Clarkson-Woodruff sketches converges to that of the Gaussian sketch as $n \to \infty$. Therefore, Theorem \ref{thm:gaussian_embedding_tw_limit} can also be used to approximate the embedding probability. Empirical studies have shown that the Hadamard and Clarkson-Woodruff sketches can give similar quality results to the Gaussian projection \citep{venkata_johnson_2011,le_2013_fastfood, dahiya_2018_empirical}. Theorem \ref{thm:data_oblivious_asymptotic_embedding} helps to characterize situations where this phenomenon is expected to be observed.
\begin{remark}
The same line of proof used in Theorem \ref{thm:data_oblivious_asymptotic_embedding} can be used to show that the convergence probability of \eqref{eq:basic_iteration} using the Hadamard and Clarkson-Woodruff projections converges to that of the Gaussian sketch under Assumption \ref{assum:leverage}. Theorem \ref{thm:gaussian_convergence_tw} also gives an asymptotic approximation for the Hadamard and Clarkson-Woodruff sketches. 
\end{remark}

It remains to establish a formal limit theorem in terms of the Tracy-Widom distribution for the Hadamard and Clarkson-Woodruff sketches. The proof of Theorem \ref{thm:data_oblivious_asymptotic_embedding} treats $k$ and $d$ as fixed, with only $n$ being taken to infinity. It is possible that Assumption \ref{assum:leverage} on the leverage scores will remain sufficient in the expanding dimension scenario. For any $d$, the maximum leverage score must be greater than the average leverage score, 
\begin{align*}
    \underset{i=1,\ldots, n}{\text{max}} \lVert \vect{u}_{(n)i} \rVert_{2}^{2} &\ge \dfrac{1}{n}\sum_{i=1}^{n} \lVert \vect{u}_{(n)i} \rVert_{2}^{2} = \dfrac{d}{n}.
\end{align*}
If we maintain that Assumption \ref{assum:leverage} holds on the leverage scores as $n,d,k \to \infty$, this implies that $d/n \to 0$. As we have assumed that our primary motivation for sketching is data compression when $n \gg d$, we feel that analysis in the asymptotic regime $d/n \to 0$ is reasonable for this use-case setting. The asymptotic approximations developed here are recommended for applications of sketching in tall-data problems $n \gg d$. 

The key result is that the Hadamard and Clarkson-Woodruff sketches behave like the Gaussian projection for large $n$, with $k$ and $d$ fixed. If the Tracy-Widom approximation in Theorem \ref{thm:gaussian_embedding_tw_limit} is good for finite $k$ and $d$ with the Gaussian sketch, it should hold well for the Hadamard and Clarkson-Woodruff projections for $n$ sufficiently large.

\subsection{Uniform sketch}
It is considerably more difficult to approximate the embedding probability for the uniform sketch compared to the other data-oblivious projections. \citet{vershynin_2010_introduction} provides a bound for the uniform sketch that is useful for comparative purposes.
\begin{theorem}[\cite{vershynin_2010_introduction}, Theorem 5.1]
\label{thm:subsampling_finite}
Consider an $n \times d$ matrix $\mat{U}$ such that $\mat{U}^{\T}\mat{U}=\mat{I}_{d}$. Let  $\vect{u}_{i}^{\T}$ represent the $i$-th row in $\mat{U}$ for $i=1, \ldots, n$. Let $m$ give an upper bound on the leverage scores, so
\begin{align*}
    \underset{i=1, \ldots, n}{\max} \ \lVert \vect{u}_{i} \rVert_{2}^{2} \le m.
\end{align*}
Let ${\mat{S}}$ be a $k \times d$ uniform sketch of size $k$. Then for every $t \ge 0$, with probability at least $1-2d\exp(-ct^2)$ one has
\begin{align*}
   1- t\sqrt{\dfrac{mn}{k}} \le \sigma_{\emph{min}}(\mat{S}\mat{U}) \le \sigma_{\emph{max}}(\mat{S}\mat{U}) \le 1+t\sqrt{\dfrac{mn}{k}}.
\end{align*}
\end{theorem}
Theorem \ref{thm:subsampling_finite} can be used to give a lower bound on the probability of obtaining an $\epsilon$-subspace embedding. Both Theorem \ref{thm:subsampling_finite} and Theorem \ref{thm:data_oblivious_asymptotic_embedding} involve the maximum leverage score. Holding $k$ and $d$ fixed, in order for the bound in Theorem \ref{thm:subsampling_finite} to remain controlled as the sample size $n$ increases, the maximum leverage score $m$ must decrease at a sufficient rate. In contrast, Assumption \ref{assum:leverage} does not enforce a rate of decay on the maximum leverage score, only that it eventually tends to zero as $n \to \infty$. This suggests that the uniform projection could be more sensitive to the maximum leverage score than the Gaussian, Hadamard and Clarkson-Woodruff projections. As mentioned earlier, it is very difficult to give a general expression for the embedding probability \eqref{eq:embedding_probability} when using the uniform sketch as it will be a complicated function of the source dataset $\mat{A}$. An advantage of the Gaussian, Hadamard and Clarkson-Woodruff projections is that a Tracy-Widom approximation can be motivated under mild regularity conditions.

\section{Data application}
\subsection{$\epsilon$-subspace embedding}
We tested the theory on a large genetic dataset of European ancestry participants in UK Biobank. The covariate data consists of genotypes at $p=1032$ genetic variants in the Protein Kinase C Epsilon (PKC$\varepsilon$) gene on $n=407,779$ subjects. Variants were filtered to have minor allele frequency of greater than one percent. The response variable was haemoglobin concentration adjusted for age, sex and technical covariates. The region was chosen as many associations with haemoglobin concentration were discovered in a genome-wide scan using univariable models; these associations were with variants with different allele frequencies, suggesting multiple distinct causal variants in the region. We also considered a subset of this dataset with $p=130$ representative markers identified by hierarchical clustering. When including the intercept and response, the PKC$\varepsilon$ subset has $n=407,779, d=132$, and the full PKC$\varepsilon$ dataset has $n=407,779, d=1034$.

The full PKC$\varepsilon$ dataset is of moderate size, so it was feasible to take the singular value decomposition of the full $n \times d$ dataset $\mat{A}=\mat{U}\mat{D}\mat{V}^{\T}$. Given the singular value decomposition we ran an oracle procedure to estimate the exact embedding probability. We generated $B$ sketching matrices $\mat{S}^{[1]}, \ldots, \mat{S}^{[B]}$. These were used to compute  $
\epsilon^{[b]} = \sigma_{\text{max}}(\mat{I}_{d}-\mat{U}^{\T}\mat{S}^{[b]\T}\mat{S}^{[b]}\mat{U})$
for $b=1, \ldots, B$ and give an estimated embedding probability as in \eqref{eq:monte_carlo_embedding}. When working with the full PKC$\varepsilon$ dataset we simulated directly from the matrix normal distribution $\widetilde{\mat{U}} \sim \text{MN}(\mat{I}_{k}, \mat{I}_{d}/k)$ for the Gaussian sketch, rather than computing the matrix multiplication $\mat{S}\mat{U}$. We took $B=1,000$ sketches of the PKC$\varepsilon$ subset, and $B=100$ sketches of the full PKC$\varepsilon$ dataset  using the uniform, Gaussian, Hadamard and Clarkson-Woodruff projections, with $k=20\times d$.

Figure \ref{fig:prkce_medium}  shows the empirical and theoretical embedding probabilities for the  PKC$\varepsilon$ subset $(n=407,779, d=132)$ for each type of sketch. The observed and theoretical curves match well for the Gaussian, Hadamard and Clarkson-Woodruff projection. The uniform projection performs worse than the other data-oblivious random projections,  as larger values of $\epsilon$ indicate weaker approximation bounds. The uniform projection does not satisfy a central limit theorem for fixed $k$, so we do not necessarily expect the Tracy-Widom law to give a good approximation for the uniform projection.  

\begin{figure}
    \centering
    \includegraphics[width=0.9\textwidth]{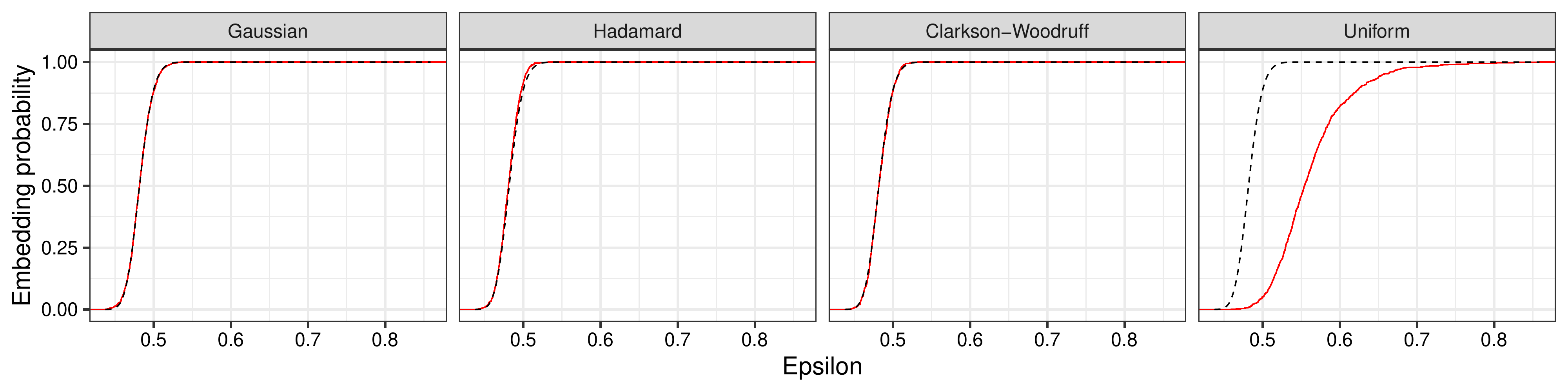}
    \caption{Analysis of subset of PKC$\varepsilon$ dataset $(n=407,779, d=132)$ with $B=1,000$ sketches of size $k=20d$. The dashed black line and the solid red line gives the theoretical and empirical embedding probabilities respectively. The Tracy-Widom approximation is accurate for the Gaussian, Hadamard and Clarkson-Woodruff sketches.}
    \label{fig:prkce_medium}
\end{figure}

Figure \ref{fig:prkce_large}  shows the empirical and theoretical embedding probabilities for the  full PKC$\varepsilon$ dataset $(n=407,779, d=1032)$ for each type of sketch. The Tracy-Widom approximation is accurate for the Gaussian sketch, but there are some deviations for the Hadamard and the Clarkson-Woodruff sketch. Interestingly, the empirical cdf for the Hadamard sketch  (red) is to the left of the theoretical value (black), indicating smaller values of $\epsilon$ than predicted. The distribution of $\epsilon$ has a longer right tail under the Clarkson-Woodruff sketch than is predicted by the Tracy-Widom law.

The deviation from the Tracy-Widom limit in Figure \ref{fig:prkce_large} could be because the finite sample approximation is poor. Theorem \ref{thm:data_oblivious_asymptotic_embedding} suggests that the Hadamard and Clarkson-Woodruff projections behave like the Gaussian sketch for $n$ sufficiently large with respect to $d$. To test this we bootstrapped the full PKC$\varepsilon$ dataset to be ten times its original size. The bootstrapped PKC$\varepsilon$ dataset has $n=4,077,790, d=1034$. We took one thousand sketches of size $k=20\times d$ using the Clarkson-Woodruff projection and ran the oracle procedure of computing $\epsilon^{[b]} = \sigma_{\text{max}}(\mat{I}_{d}-\mat{U}^{\mathsf{T}}\mat{S}^{[b]\mathsf{T}}\mat{S}^{[b]}\mat{U})$ for each sketch. Figure \ref{fig:prkce_bootstrap}  compares the distribution of $\sigma_{\text{max}}(\mat{I}_{d}-\mat{U}^{\mathsf{T}}\mat{S}^{\mathsf{T}}\mat{S}\mat{U})$ using Clarkson-Woodruff projection on the original dataset and on the large bootstrapped dataset. As $n$ increases we expect the quality of the Tracy-Widom approximation to improve. Panel (a) of Figure \ref{fig:prkce_bootstrap} compares the theoretical to the simulation results on the original dataset. The Clarkson-Woodruff projection shows greater variance than expected.  Panel (b) compares the theoretical to the simulation results on the bootstrapped dataset. In (b) there is very good agreement between the empirical distribution and the theoretical distribution. It seems that for this dataset $n \approx 400,000$ is not big enough for the large sample asymptotics to kick in. At $n \approx 4$ million the Tracy-Widom approximation is very good. As mentioned earlier, our motivation for using a sketching algorithm is to perform data compression with tall datasets $n \gg d$. This example highlights that the asymptotic approximations become more accurate as the sample size $n$ grows while the computational incentives for using sketching increase in parallel.

\begin{table}
\centering
\begin{tabular}{@{} l  l l @{}} 
\toprule
Projection & Subset $(p=132)$ & Full $(p=1034)$ \\ 
\midrule 
Gaussian & 769 &  -  \\
Hadamard & 17.2&  156 \\
Clarkson-Woodruff & 1.33 & 21 \\
Uniform & 0.03& 2.8 \\
\bottomrule
\end{tabular}
\caption{Mean sketching time (seconds) over ten sketches for each dataset. The Gaussian sketch is considerably slower than the Hadamard and Clarkson-Woodruff sketches on the subset as is expected from Table \ref{tab:run_time} }
\label{tab:prkce_epsilon_medium_times} 
\end{table}

\begin{figure}
    \centering
    \includegraphics[width=0.9\textwidth]{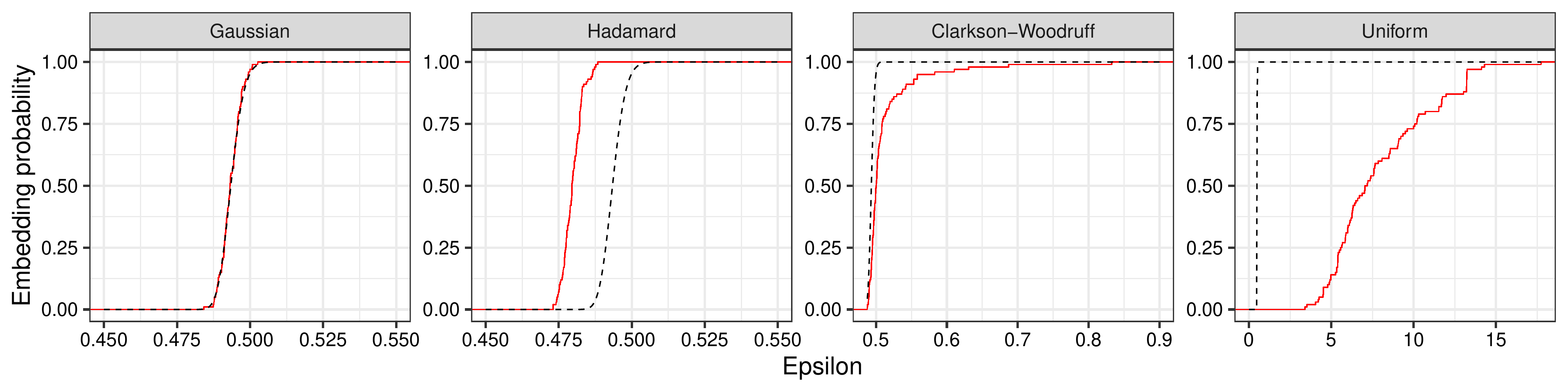}
    \caption{Analysis of full PKC$\varepsilon$ dataset $(n=407,779, d=1,034)$ with $B=100$ sketches of size $k=20d$. The $x$-axis is different in each panel.The dashed black line and the solid red line gives the theoretical and empirical embedding probabilities respectively. The Uniform projection is much less successful at generating $\epsilon$-subspace embeddings than the other data-oblivious projections.}
    \label{fig:prkce_large}
\end{figure}

\begin{figure}
    \centering
    \includegraphics[width=0.6\textwidth]{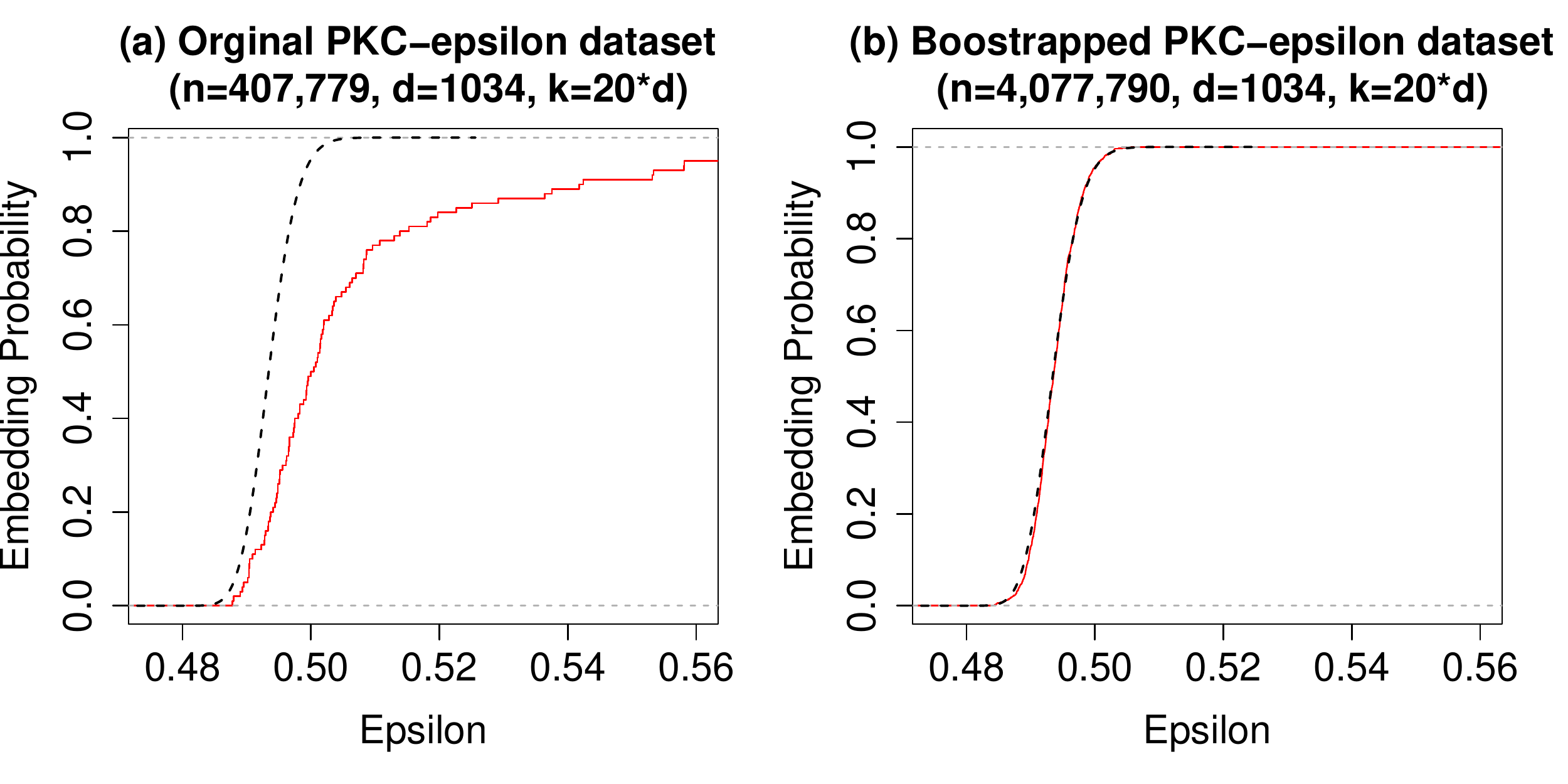}
    \caption{Comparison of results on the original PKC$\varepsilon$ dataset ($n=407,779$) and the bootstrapped larger PKC$\epsilon$ dataset ($n=4,077,790$). The dashed black line and the solid red line gives the theoretical and empirical embedding probabilities respectively. As expected from Theorem \ref{thm:data_oblivious_asymptotic_embedding},  the accuracy of the Tracy-Widom increases with $n$.}
    \label{fig:prkce_bootstrap}
\end{figure}

\subsection{Iterative optimisation}
We considered iterative least-squares optimisation using the song year dataset available from the UCI machine learning repository. The dataset has $n=515,344$ observations, $p=90$ covariates, and year of song release as the response.  We assessed the convergence probability by running the iteration \eqref{eq:basic_iteration} with the sketched preconditioner.  The initial parameter estimate $\vect{\beta}^{(0)}$ was a vector of zeros. The iteration was run for 2000 steps, with convergence being declared if the gradient norm  condition $\lVert \mat{X}^{\T}(\vect{y}-\vect{X}\vect{\beta}^{(t)})\rVert_{2} < 10^{-6}$ was satisfied any time step $t$. This convergence criterion was used instead of $\lVert\vect{\beta}_{F} - \vect{\beta}^{(t)} \rVert_{2}$ as $\vect{\beta}_{F}$ will not be known in practice. This was repeated one hundred times for each of the random projections discussed in Section \ref{subsec:sketching} using different sketch sizes $k$. Figure \ref{fig:year_convergence} compares the empirical (black solid points) and theoretical convergence probabilities (dashed red line) against the sketch size $k$. The point-ranges represent 95\% confidence intervals. The Gaussian, Hadamard and Clarkson-Woodruff show near identical behavior, and the empirical convergence probabilities closely match the theoretical predictions using Theorem \ref{thm:gaussian_convergence_tw}. The uniform sketch was much less successful in generating preconditioners, the algorithm did not show convergence in any replication at each sketch size $k$. In this example, the additional computational cost of the Gaussian, Hadamard and Clarkson-Woodruff sketches compared to the Uniform subsampling has clear benefits.

\begin{figure}
    \centering
\includegraphics[width=0.9\textwidth]{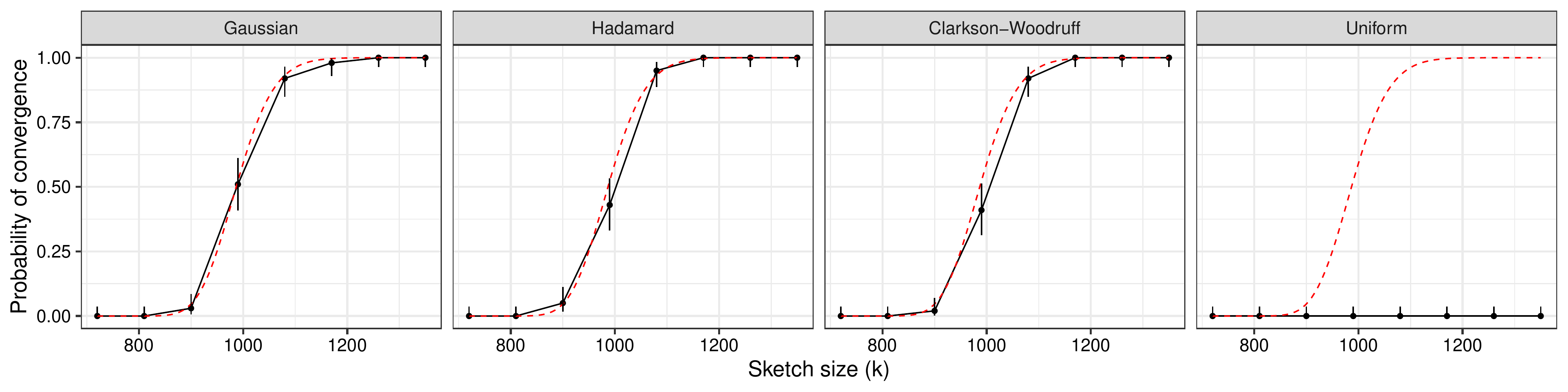}
\caption{Convergence probability on year dataset $(n=515,344, d=91)$. Black solid points show the empirical convergence probability over $B=100$ sketches. The red dashed line gives the theoretical convergence probability using Theorem \ref{thm:gaussian_convergence_tw}. The Tracy-Widom approximation is accurate for the Gaussian, Hadamard and Clarkson-Woodruff sketches. The uniform sketch fails to generate useful preconditioners. }
    \label{fig:year_convergence}
\end{figure}

\section{Conclusion}
The analysis of the asymptotic behavior of common data-oblivious random projections revealed an important connection to the Tracy-Widom law. The probability of attaining an $\epsilon$-subspace embedding (Definition \ref{defn:epsilon_subspace_embedding}) is an integral descriptive measure for many sketching algorithms. The asymptotic embedding probability can approximated using the Tracy-Widom law for the Gaussian, Hadamard and Clarkson-Woodruff sketches. The Tracy-Widom law can also be used to estimate the convergence probability for iterative schemes with a sketched preconditioner. We have tested the predictions empirically  and seen close agreement. The majority of existing results for sketching algorithms have been established using non-asymptotic tools. Asymptotic results are a useful complement that can provide answers to important questions that are difficult to address concretely in a finite dimensional framework. 

There was a stark contrast between the performance of the basic uniform projection and the other data-oblivious projections (Gaussian, Hadamard and Clarkson-Woddruff) in the data application. The Hadmard and Clarkson-Woodruff projections are expected to behave like the Gaussian projection under mild regularity conditions on the maximum leverage score. We observed this phenomenon when $n/d$ was large, as is required by Theorem \ref{thm:data_oblivious_asymptotic_embedding}. The Hadamard and Clarkson-Woodruff projections are substantially more computationally efficient than the Gaussian projection (recall Table \ref{tab:run_time}), so their universal limiting behavior implies that the trade-off between computation time and performance guarantees is asymptotically negligible in the regime \eqref{eq:big_data_regime}. 

The Tracy-Widom law has found many applications in high-dimensional statistics and probability \citep{edelman_2013_random}, and we have shown that it useful for describing the asymptotic behavior of sketching algorithms. The asymptotic behaviour with respect to large $n$ is of practical interest, as this is the regime where sketching is attractive as a data compression technique. The universal behavior of high-dimensional random matrices has practical and theoretical consequences for randomized algorithms that use linear dimension reduction \citep{dobriban_2018_new, lacotte_2020_limiting}. 


\bibliographystyle{spcustom}      
\bibliography{bibliography}   

\newcommand{\noop}[1]{}
\begin{thebibliography}{50}
\providecommand{\natexlab}[1]{#1}
\providecommand{\url}[1]{\texttt{#1}}
\providecommand{\urlprefix}{URL }
\expandafter\ifx\csname urlstyle\endcsname\relax
  \providecommand{\doi}[1]{doi:\discretionary{}{}{}#1}\else
  \providecommand{\doi}{doi:\discretionary{}{}{}\begingroup
  \urlstyle{rm}\Url}\fi
\providecommand{\selectlanguage}[1]{\relax}

\bibitem[{Ahfock et~al.(2020)Ahfock, Astle and
  Richardson}]{ahfock_statistical_2020}
Ahfock, D.C., Astle, W.J., Richardson, S.: {Statistical properties of sketching
  algorithms}.
\newblock Biometrika \textbf{108}(2), 283--297 (2020)

\bibitem[{Ailon and Chazelle(2009)}]{ailon_fast_2009}
Ailon, N., Chazelle, B.: The fast {Johnson} {Lindenstrauss} transform and
  approximate nearest neighbors.
\newblock SIAM Journal on Computing \textbf{39}(1), 302--322 (2009)

\bibitem[{Bai and Silverstein(2010)}]{bai_spectral_2010}
Bai, Z., Silverstein, J.W.: Spectral Analysis of Large Dimensional Random
  Matrices.
\newblock Springer, New York, 2nd ed. (2010)

\bibitem[{Bardenet and Maillard(2015)}]{bardenet_note_2015}
Bardenet, R., Maillard, O.A.: A note on replacing uniform subsampling by random
  projections in {MCMC} for linear regression of tall datasets.
\newblock HAL preprint 01248841  (2015)

\bibitem[{Bhatia(1996)}]{bhatia_matrix_1996}
Bhatia, R.: Matrix Analysis.
\newblock Springer (1996)

\bibitem[{Billingsley(1999)}]{billingsley_convergence_1999}
Billingsley, P.: Convergence of Probability Measures.
\newblock Wiley Series in Probability and Statistics. Wiley, New York, 2nd ed.
  (1999)

\bibitem[{Chiani(2017)}]{chiani_2017_probability}
Chiani, M.: On the probability that all eigenvalues of {G}aussian, {W}ishart,
  and double {W}ishart random matrices lie within an interval.
\newblock IEEE Transactions on Information Theory \textbf{63}(7), 4521--4531
  (2017)

\bibitem[{Clarkson and Woodruff(2013)}]{clarkson_low_2013}
Clarkson, K.L., Woodruff, D.P.: Low rank approximation and regression in input
  sparsity time.
\newblock In: Proceedings of the forty-fifth annual {ACM} symposium on {Theory}
  of {Computing}, pp. 81--90. ACM (2013)

\bibitem[{Cormode(2011)}]{cormode_sketch_2011}
Cormode, G.: Sketch techniques for approximate query processing.
\newblock Foundations and Trends in Databases  (2011)

\bibitem[{Dahiya et~al.(2018)Dahiya, Konomis and
  Woodruff}]{dahiya_2018_empirical}
Dahiya, Y., Konomis, D., Woodruff, D.P.: An empirical evaluation of sketching
  for numerical linear algebra.
\newblock In: Proceedings of the 24th ACM SIGKDD International Conference on
  Knowledge Discovery \& Data Mining, pp. 1292--1300. ACM (2018)

\bibitem[{{Dobriban} and {Liu}(2018)}]{dobriban_2018_new}
{Dobriban}, E., {Liu}, S.: A new theory for sketching in linear regression.
\newblock arXiv preprint arXiv:1810.06089 (2018)

\bibitem[{Drineas et~al.(2006)Drineas, Mahoney and
  Muthukrishnan}]{drineas_sampling_2006}
Drineas, P., Mahoney, M.W., Muthukrishnan, S.: Sampling algorithms for l2
  regression and applications.
\newblock In: Proceedings of the seventeenth annual ACM-SIAM symposium on
  Discrete algorithms, pp. 1127--1136. Society for Industrial and Applied
  Mathematics (2006)

\bibitem[{Edelman(1988)}]{edelman_eigenvalues_1988}
Edelman, A.: Eigenvalues and condition numbers of random matrices.
\newblock SIAM Journal on Matrix Analysis and Applications \textbf{9}(4),
  543--560 (1988)

\bibitem[{Edelman and Wang(2013)}]{edelman_2013_random}
Edelman, A., Wang, Y.: Random matrix theory and its innovative applications.
\newblock In: Advances in Applied Mathematics, Modeling, and Computational
  Science, pp. 91--116. Springer (2013)

\bibitem[{Erichson et~al.(2016)Erichson, Voronin, Brunton and
  Kutz}]{erichson_randomized_2016}
Erichson, N.B., Voronin, S., Brunton, S.L., Kutz, J.N.: Randomized {Matrix}
  {Decompositions} using {R}.
\newblock arXiv preprint p. arXiv:1608.02148 (2016)

\bibitem[{Falcone et~al.(2021)Falcone, Anderlucci and
  Montanari}]{falcone_2021_matrix}
Falcone, R., Anderlucci, L., Montanari, A.: Matrix sketching for supervised
  classification with imbalanced classes.
\newblock Data Mining and Knowledge Discovery pp. 1--35 (2021)

\bibitem[{Geman(1980)}]{geman_limit_1980}
Geman, S.: A limit theorem for the norm of random matrices.
\newblock The Annals of Probability \textbf{8}(2), 252--261 (1980)

\bibitem[{Geppert et~al.(2017)Geppert, Ickstadt, Munteanu, Quedenfeld and
  Sohler}]{geppert_random_2017}
Geppert, L.N., Ickstadt, K., Munteanu, A., Quedenfeld, J., Sohler, C.: Random
  projections for {B}ayesian regression.
\newblock Statistics and Computing \textbf{27}(1), 79--101 (2017)

\bibitem[{Grellmann et~al.(2016)Grellmann, Neumann, Bitzer, Kovacs, Tönjes,
  Westlye, Andreassen, Stumvoll, Villringer and
  Horstmann}]{grellmann_random_2016}
Grellmann, C., Neumann, J., Bitzer, S., Kovacs, P., Tönjes, A., Westlye, L.T.,
  Andreassen, O.A., Stumvoll, M., Villringer, A., Horstmann, A.: Random
  {Projection} for {Fast} and {Efficient} {Multivariate} {Correlation}
  {Analysis} of {High}-{Dimensional} {Data}: {A} {New} {Approach}.
\newblock Frontiers in Genetics \textbf{7}, 102 (2016)

\bibitem[{Hageman and Young(2012)}]{hageman_2012_applied}
Hageman, L., Young, D.: Applied Iterative Methods.
\newblock Dover Books on Mathematics. Dover Publications (2012)

\bibitem[{Halko et~al.(2011)Halko, Martinsson and Tropp}]{halko_finding_2011}
Halko, N., Martinsson, P.G., Tropp, J.A.: Finding structure with randomness:
  Probabilistic algorithms for constructing approximate matrix decompositions.
\newblock SIAM Review \textbf{53}(2), 217--288 (2011)

\bibitem[{Johnstone(2001)}]{johnstone_distribution_2001}
Johnstone, I.M.: On the distribution of the largest eigenvalue in principal
  components analysis.
\newblock Annals of Statistics pp. 295--327 (2001)

\bibitem[{Johnstone(2006)}]{johnstone_high_2006}
Johnstone, I.M.: High dimensional statistical inference and random matrices.
\newblock arXiv preprint arXiv:0611589  (2006)

\bibitem[{Johnstone et~al.(2014)Johnstone, Ma, Perry and
  Shahram}]{johnstone_2014_rmtstat}
Johnstone, I.M., Ma, Z., Perry, P.O., Shahram, M.: RMTstat: Distributions,
  Statistics and Tests derived from Random Matrix Theory (2014).
\newblock R package version 0.3

\bibitem[{Lacotte et~al.(2020)Lacotte, Liu, Dobriban and
  Pilanci}]{lacotte_2020_limiting}
Lacotte, J., Liu, S., Dobriban, E., Pilanci, M.: Limiting spectrum of
  randomized {H}adamard transform and optimal iterative sketching methods.
\newblock arXiv preprint arXiv:2002.00864  (2020)

\bibitem[{Le et~al.(2013)Le, Sarl{\'o}s and Smola}]{le_2013_fastfood}
Le, Q., Sarl{\'o}s, T., Smola, A.: Fastfood-computing hilbert space expansions
  in loglinear time.
\newblock In: International Conference on Machine Learning, pp. 244--252 (2013)

\bibitem[{Lehoucq et~al.(1998)Lehoucq, Sorensen and Yang}]{lehoucq_arpack_1998}
Lehoucq, R.B., Sorensen, D.C., Yang, C.: ARPACK users' guide: solution of
  large-scale eigenvalue problems with implicitly restarted Arnoldi methods,
  vol.~6.
\newblock {SIAM} (1998)

\bibitem[{Li et~al.(2006)Li, Hastie and Church}]{li_very_2006}
Li, P., Hastie, T.J., Church, K.W.: Very sparse random projections.
\newblock In: Proceedings of the 12th {ACM} {SIGKDD} international conference
  on {Knowledge} discovery and data mining, pp. 287--296. ACM (2006)

\bibitem[{{Lopes} et~al.(2018){Lopes}, {Wang} and {Mahoney}}]{lopes_2018_error}
{Lopes}, M.E., {Wang}, S., {Mahoney}, M.W.: {Error Estimation for Randomized
  Least-Squares Algorithms via the Bootstrap}.
\newblock arXiv preprint arXiv:1803.08021 (2018)

\bibitem[{Ma et~al.(2015)Ma, Mahoney and Yu}]{ma_statistical_2015}
Ma, P., Mahoney, M.W., Yu, B.: A statistical perspective on algorithmic
  leveraging.
\newblock Journal of Machine Learning Research \textbf{16}(1), 861--911 (2015)

\bibitem[{Ma(2012)}]{ma_accuracy_2012}
Ma, Z.: Accuracy of the {T}racy--{W}idom limits for the extreme eigenvalues in
  white {W}ishart matrices.
\newblock Bernoulli \textbf{18}(1), 322--359 (2012)

\bibitem[{Mahoney(2011)}]{mahoney_randomized_2011}
Mahoney, M.: Randomized algorithms for matrices and data.
\newblock Foundations and Trends in Machine Learning \textbf{3}(2), 123--224
  (2011)

\bibitem[{Mahoney and Drineas(2016)}]{mahoney_structural_2016}
Mahoney, M., Drineas, P.: Structural properties underlying high-quality
  {Randomized} {Numerical} {Linear} {Algebra} algorithms.
\newblock In: Buhlmann, P., Drineas, P., Kane, M., van~de Laan, M. (eds.)
  Handbook of {Big} {Data}, pp. 137--154. Chapman and Hall (2016)

\bibitem[{Meng(2014)}]{meng_randomized_2014}
Meng, X.: Randomized {Algorithms} for {Large}-scale {Strongly}
  {Over}-determined {Linear} {Regression} {Problems}.
\newblock Ph.D. thesis, Stanford University, Stanford, California, United
  States (2014)

\bibitem[{Meng and Mahoney(2013)}]{meng_low-distortion_2013}
Meng, X., Mahoney, M.M.: Low-distortion {Subspace} {Embeddings} in
  {Input}-sparsity {Time} and {Applications} to {Robust} {Linear} {Regression}.
\newblock In: Proceedings of the forty-fifth annual {ACM} symposium on {Theory}
  of {computing}, pp. 91--100. ACM (2013)

\bibitem[{Meng et~al.(2014)Meng, Saunders and Mahoney}]{meng_2014_lsrn}
Meng, X., Saunders, M.A., Mahoney, M.W.: Lsrn: A parallel iterative solver for
  strongly over-or underdetermined systems.
\newblock SIAM Journal on Scientific Computing \textbf{36}(2), C95--C118 (2014)

\bibitem[{Nelson and Nguy{\^e}n(2013)}]{nelson_osnap_2013}
Nelson, J., Nguy{\^e}n, H.L.: Osnap: Faster numerical linear algebra algorithms
  via sparser subspace embeddings.
\newblock In: 54th Annual {IEEE} Symposium on the {F}oundations of {C}omputer
  {S}cience, pp. 117--126. IEEE (2013)

\bibitem[{Pilanci and Wainwright(2016)}]{pilanci_iterative_2016}
Pilanci, M., Wainwright, M.J.: Iterative {H}essian sketch: Fast and accurate
  solution approximation for constrained least-squares.
\newblock Journal of Machine Learning Research \textbf{17}(1), 1842--1879
  (2016)

\bibitem[{Quiroz et~al.(2018)Quiroz, Villani, Kohn, Tran and
  Dang}]{quiroz_2018_subsampling}
Quiroz, M., Villani, M., Kohn, R., Tran, M.N., Dang, K.D.: Subsampling
  {M}{C}{M}{C}-an introduction for the survey statistician.
\newblock Sankhya A \textbf{80}(1), 33--69 (2018)

\bibitem[{Raskutti and Mahoney(2014)}]{raskutti_statistical_2014}
Raskutti, G., Mahoney, M.: A {Statistical} {Perspective} on {Randomized}
  {Sketching} for {Ordinary} {Least}-{Squares}.
\newblock arXiv preprint arXiv:1406.5986  (2014)

\bibitem[{{Roosta-Khorasani} and {Mahoney}(2016)}]{roosta_subsampled_2016}
{Roosta-Khorasani}, F., {Mahoney}, M.W.: {Sub-Sampled {N}ewton Methods I:
  Globally Convergent Algorithms}.
\newblock arXiv preprint arXiv:1601.04737  (2016)

\bibitem[{Sarlos(2006)}]{sarlos_improved_2006}
Sarlos, T.: Improved approximation algorithms for large matrices via random
  projections.
\newblock In: 47th {Annual} {IEEE} {Symposium} on {Foundations} of {Computer}
  {Science}, pp. 143--152. IEEE (2006)

\bibitem[{Silverstein(1985)}]{silverstein_smallest_1985}
Silverstein, J.W.: The smallest eigenvalue of a large dimensional wishart
  matrix.
\newblock The Annals of Probability \textbf{13}(4), 1364--1368 (1985)

\bibitem[{Tracy and Widom(1994)}]{tracy_level_1994}
Tracy, C.A., Widom, H.: Level-spacing distributions and the airy kernel.
\newblock Communications in Mathematical Physics \textbf{159}(1), 151--174
  (1994)

\bibitem[{Tropp(2011)}]{tropp_improved_2011}
Tropp, J.A.: Improved analysis of the subsampled randomized {H}adamard
  transform.
\newblock Advances in Adaptive Data Analysis \textbf{3}(01n02), 115--126 (2011)

\bibitem[{Van Der~Vaart(1998)}]{van_der_vaart_asymptotic_1998}
Van Der~Vaart, A.: Asymptotic {Statistics}.
\newblock Cambridge {Series} in {Statistical} and {Probabilistic}
  {Mathematics}, 3. Cambridge University Press (1998)

\bibitem[{Venkatasubramanian and Wang(2011)}]{venkata_johnson_2011}
Venkatasubramanian, S., Wang, Q.: The {J}ohnson-{L}indenstrauss transform: an
  empirical study.
\newblock In: 2011 Proceedings of the Thirteenth Workshop on Algorithm
  Engineering and Experiments, pp. 164--173. SIAM (2011)

\bibitem[{Vershynin(2010)}]{vershynin_2010_introduction}
Vershynin, R.: Introduction to the non-asymptotic analysis of random matrices.
\newblock arXiv preprint arXiv:1011.3027  (2010)

\bibitem[{Woodruff(2014)}]{woodruff_sketching_2014}
Woodruff, D.P.: Sketching as a tool for numerical linear algebra.
\newblock Foundations and Trends in Theoretical Computer Science
  \textbf{10}(1-2), 1--157 (2014)

\bibitem[{Yang et~al.(2015)Yang, Meng and Mahoney}]{yang_implementing_2015}
Yang, J., Meng, X., Mahoney, M.W.: Implementing randomized matrix algorithms in
  parallel and distributed environments.
\newblock arXiv preprint arXiv:1502.03032  (2015)

\end{thebibliography}

\makeatletter
\renewcommand \thesection{S\@arabic\c@section}
\renewcommand\thetable{S\@arabic\c@table}
\renewcommand \thefigure{S\@arabic\c@figure}
\renewcommand{\theequation}{S.\arabic{equation}}
\renewcommand{\thefigure}{S\arabic{figure}}
\makeatother

\def\theassumption{S.\arabic{assumption}}
\def\thecorollary{S.\arabic{corollary}}
\def\thelemma{S.\arabic{lemma}}
\def\theproposition{S.\arabic{proposition}}
\def\thetheorem{S.\arabic{theorem}}
\def\thedefinition{S.\arabic{definition}}

\setcounter{section}{1}
\section*{Supplementary Material}
\subsection{Weak convergence}
Our asymptotic arguments concern the convergence of sequences of probability measures. \cite{billingsley_convergence_1999} is an authoritative reference on the topic. We now recap some useful foundational theory, as is presented in \citet[][Chapter 2]{van_der_vaart_asymptotic_1998}. The Portmanteau lemma gives a number of useful equivalent definitions of convergence in distribution (weak convergence). 
\begin{lemma}[Portmanteau]
\label{lem:portmanteau}
Let $(\vect{Z}_{n})_{n \in \mathbb{N}}$ denote a sequence of random vectors of fixed dimension, and $\vect{Z}$ denote another random vector of the same dimension.  The following statements are equivalent, where limits are being taken in $n$:
\begin{enumerate}[(a)]
\item $\Pr(\vect{Z}_{n} \le \vect{z}) \to \Pr(\vect{Z} \le \vect{z})$ at all continuity points $\vect{z}$ of the cumulative distribution function $\Pr(\vect{Z} \le \vect{z})$.
\item \label{prop:boundary} $\Pr(\vect{Z}_{n} \in B) \to \Pr(\vect{Z} \in B)$ for all Borel sets $B$ with $\Pr(\vect{Z} \in \partial B)=0$, where $\partial B$ denotes the boundary of the set $B$. The boundary is defined as the closure of the set $B$ minus the interior of $B$, so $\partial B = \overline{B} \setminus B^{o}$.
\end{enumerate}
\end{lemma}

\begin{lemma}[Uniform convergence]
\label{lem:uniform}
Suppose that $(\vect{Z}_{n})$ converges in distribution to a random vector $\vect{Z}$ with a continuous distribution function. Then \begin{align*}
  \underset{n \to \infty}{\lim}  \underset{\vect{z}}{\sup}\left\vert \Pr(\vect{Z}_{n} \le \vect{z})  - \Pr(\vect{Z} \le \vect{z}) \right\vert = 0.
\end{align*}
\end{lemma}
Proofs for these results are given in Chapter 2 of \cite{van_der_vaart_asymptotic_1998}. 

\subsection{Random Matrix Theory}

\begin{definition}
A random variable $Z$ has a Tracy-Widom distribution $F_{1}$, when the cumulative distribution function is given by
\begin{align*}
F_{1}(z) = \exp \left(-\dfrac{1}{2} \int_{z}^{\infty}q({t}) + ({t}-{z})q^2({t}) \ d{t} \right).
\end{align*}
Where $q({z})$ satisfies the nonlinear differential equation $q''({z}) = {z}q(z) + 2q^3(z)$, subject to the asymptotic boundary condition, $
q(z) \sim \text{Ai}(z) \quad \text{as } z  \to \infty$. The function $\text{Ai}(z)$ denotes the Airy function, defined as
$
\text{Ai}(z) = {\pi}^{-1}\int_{0}^{\infty} \cos \left(t^3/3 + zt\right) \  dt$.
\end{definition}

\begin{theorem}{\citep{ma_accuracy_2012}}
\label{thm:tw_maximum_eigenvalue}
\newline
Consider a sequence of $\emph{Wishart}(k, \mat{I}_{d}/k)$ random matrices where $d,k \to \infty$ and $d/k \to \alpha$ with $\alpha \in (0, 1]$. Let $\lambda_{\emph{max}}$ denote the maximum eigenvalue of the random matrix. Define the centering and scaling constants as
\begin{align*}
\mu_{k,d} &= k^{-1}(\sqrt{k-1/2}+\sqrt{d-1/2})^{2}, \quad \sigma_{k,d} = \dfrac{k^{-1}(\sqrt{k-1/2}+\sqrt{d-1/2})}{\left(1/{\sqrt{k-1/2}}+1/{\sqrt{d-1/2}}\right)^{1/3}}. 
\end{align*}
Then with  $Z \sim F_{1}$ and $F_{1}$ is the Tracy-Widom distribution.
\begin{align*}
\dfrac{(\lambda_{\emph{max}}-\mu_{k,d})}{\sigma_{k,d}} \overset{d}{\to} Z.
\end{align*}

\end{theorem}
A limit theorem for the minimum eigenvalue is best expressed in terms of the logarithm of the minimum eigenvalue as this gives higher order accuracy \citep{ma_accuracy_2012}.
\begin{theorem}{\citep{ma_accuracy_2012}}
\label{thm:tw_minimum_eigenvalue}
\newline
Consider a sequence of $\emph{Wishart}(k, \mat{I}_{d}/k)$ random matrices where $d,k \to \infty$ and $d/k \to \alpha$ with $\alpha \in (0, 1]$.  Let $\lambda_{\emph{min}}$ denote the minimum eigenvalue of the random matrix. Set
\begin{align*}
\mu_{k,d} &= (\sqrt{k-1/2}-\sqrt{d-1/2})^{2},\\
\sigma_{k,d} &= (\sqrt{k-1/2}-\sqrt{d-1/2})\left(\dfrac{1}{\sqrt{k-1/2}}-\dfrac{1}{\sqrt{d-1/2}}\right)^{1/3},
\end{align*}
and define the following centering and scaling constants $ \tau_{k,d} = \sigma_{k,d}/\mu_{k,d},  \nu_{k,d} = \log(\mu_{k,d})-\log k -\tau_{k,d}^2/8$. Then where  $Z \sim F_{1}$ and $F_{1}$ is the Tracy-Widom distribution,
\begin{align*}
\dfrac{(\log \lambda_{\emph{min}} -\nu_{k,d})}{\tau_{k,d}} \overset{d}{\to} -Z,
\end{align*}
\end{theorem}

\subsection{Proof of Theorem 1}
\begin{proof} The extreme eigenvalues of a Wishart random matrix converge in probability to fixed values as both the dimension and degrees of freedom expand. The result for the largest eigenvalue is due to \cite{geman_limit_1980} and the result for the smallest eigenvalue is due to \cite{silverstein_smallest_1985}.
\begin{theorem}{\citep{geman_limit_1980, silverstein_smallest_1985}}
\label{thm:pointwise_wishart}
\newline
Consider a sequence of $\emph{Wishart}(k, \mat{I}_{d}/k)$ random matrices where the degrees of freedom $k$ and dimension $d$ are both taken to infinity. Suppose that the variables to {samples} ratio $d/k$ converges to a constant $(d/k) \to \alpha$, where $\alpha \in (0, 1]$. Then the extreme eigenvalues of the random matrix, $\lambda_{\emph{min}}$ and $\lambda_{\emph{max}}$ converge in probability to the limits
\begin{align}
(i) \ \lambda_{\emph{min}} \overset{p}{\to} (1-\sqrt{\alpha})^2, \\
(ii) \ \lambda_{\emph{max}} \overset{p}{\to} (1+\sqrt{\alpha})^2.
\end{align}
\end{theorem}
Theorem \ref{thm:pointwise_wishart} and the continuous mapping theorem  can be used to determine the asymptotic embedding probability for the Gaussian sketch. 

\begin{lemma}
\label{lem:gaussian_pointwise_limit}
Suppose we have an arbitrary $n \times d$ data matrix $\mat{A}$ where $n >d$ and $\mat{A}$ is of rank $d$. Assume we take a Gaussian sketch of size $k$. Then asymptotically in $n, k$ and $d$, with $d/k \to \alpha$ where $\alpha \in (0, 1]$, 
\begin{align*}
\underset{n,d,k \to \infty}{\lim} \Pr (\mat{S} \emph{ is an $\epsilon$-subspace embedding for $\mat{A}_{(n)}$}) =  \begin{cases}
0 & \mbox{ \emph{if }} \epsilon < (1+\sqrt{\alpha})^2-1 \\
1 & \mbox{ \emph{if }} \epsilon > (1+\sqrt{\alpha})^2-1
\end{cases}
\end{align*}
\end{lemma}

\begin{proof} Let $\mat{W} \sim \text{Wishart}(k, \mat{I}_{d}/k)$, and let $\lambda_{\text{min}}$ and $\lambda_{\text{max}}$ denote the minimum and maximum eigenvalues of $\mat{W}$ respectively. Using Slutsky's theorem and the continuous mapping theorem we have the joint convergence result
\begin{align}
\begin{bmatrix}
\lvert 1-\lambda_{\text{min}}\rvert  \\
\lvert1-\lambda_{\text{max}}\rvert
\end{bmatrix}
\overset{p}{\to}
\begin{bmatrix}
\lvert 1-(1-\sqrt{\alpha})^2\rvert\\
\lvert 1-(1+\sqrt{\alpha})^2\rvert
\end{bmatrix} = \begin{bmatrix}
 2\sqrt{\alpha} - \alpha \\
 2\sqrt{\alpha} + \alpha
\end{bmatrix}, \label{eq:joint_pointwise_convergence}
\end{align}
where the equality uses the fact that $\alpha \in (0, 1]$. For large $k$ and $d$, the maximum eigenvalue $\lambda_{\text{max}}$ is expected to show greater deviation from one than the minimum eigenvalue $\lambda_{\text{min}}$. Over the interval $\alpha \in (0, 1]$ it holds that
\begin{align*}
\lvert 1-(1+\sqrt{\alpha})^2 \rvert  > \lvert1-(1-\sqrt{\alpha})^2 \rvert.
\end{align*}
Applying the continuous mapping theorem to the random vector in \eqref{eq:joint_pointwise_convergence},
\begin{align*}
\text{max}
\begin{bmatrix}
\lvert 1-\lambda_{\text{min}}\rvert  \\
\lvert1-\lambda_{\text{max}}\rvert
\end{bmatrix}
\overset{p}{\to}
\text{max}
\begin{bmatrix}
\lvert 1-(1-\sqrt{\alpha})^2\rvert\\
\lvert 1-(1+\sqrt{\alpha})^2\rvert
\end{bmatrix},
\end{align*}
yields $\text{max}(| 1- \lambda_{\text{min}}| , | 1- \lambda_{\text{max}}| ) \overset{p}{\to} \lvert 1-(1+\sqrt{\alpha})^2 \rvert$. Now as $(1+\sqrt{\alpha})^2$ is greater than one for all $\alpha > 0$, the absolute value sign can be removed in the limit giving the equivalent statement $
\text{max}(| 1- \lambda_{\text{min}}| , | 1- \lambda_{\text{max}}| ) \overset{p}{\to}  (1+\sqrt{\alpha})^2-1$. Recalling that  $\sigma_{\text{max}}(\mat{I}_{d}-\mat{W})= \text{max}(| 1- \lambda_{\text{min}}| , | 1- \lambda_{\text{max}}| )$, we establish convergence of the limiting singular value
\begin{align}
\sigma_{\text{max}}(\mat{I}_{d}-\mat{W})  \overset{p}{\to}(1+\sqrt{\alpha})^2-1. \label{eq:singular_value_pointwise_limit}
\end{align}
As convergence in probability to a constant implies convergence in distribution, the Portmanteau lemma then gives the probabilistic statement
\begin{align*}
\underset{n,d,k \to \infty}{\lim} \Pr(\sigma_{\text{max}}(\mat{I}_{d}-\mat{W})  \le \epsilon)  &= \begin{cases}
0 & \mbox{ \text{if }} \epsilon < (1+\sqrt{\alpha})^2-1, \\
1 & \mbox{ \text{if }} \epsilon > (1+\sqrt{\alpha})^2-1.
\end{cases}
\end{align*}
As $\epsilon = (1+\sqrt{\alpha})^{2}-1$ is a discontinuity point of the limiting  distribution function we do not make a statement about the case $\epsilon = (1+\sqrt{\alpha})^{2}-1$. 
We have the equality in limits
\begin{align*}
\underset{n,d,k \to \infty}{\lim} \Pr (\mat{S} \text{ is an $\epsilon$-subspace embedding for $\mat{A}$}) &= \underset{n,d,k \to \infty}{\lim} \Pr(\sigma_{\text{max}}(\mat{I}_{d}-\mat{W})  \le \epsilon),
\end{align*}
giving the final result.
\end{proof}
Given Lemma \ref{lem:gaussian_pointwise_limit}, we can move on to the proof of Theorem 1. Let $\mat{W} \sim \text{Wishart}(k, \mat{I}_{d}/k)$, and let $\lambda_{\text{min}}$ and $\lambda_{\text{max}}$ denote the minimum and maximum eigenvalues of $\mat{W}$ respectively. The majority of the proof comes down to showing that $\lambda_{\text{max}}$ controls the embedding probability. Using the Portmanteau lemma (Lemma \ref{lem:portmanteau}) we will show that
\begin{align*}
\underset{d,k \to \infty}{\lim} \Pr(\sigma_{\text{max}}(\mat{I}_{d}-\mat{W}) \le \epsilon)  &=  \underset{d,k \to \infty}{\lim} \Pr \left( | 1- \lambda_{\text{max}}|  \le \epsilon \right).
\end{align*}
Recall the key expression 
\begin{align*}
\Pr (\mat{S} \text{ is an $\epsilon$-subspace embedding for $\mat{A}$}) &= \Pr(\sigma_{\text{max}}(\mat{I}_{d}-\mat{W}) \le \epsilon) \\
&= \Pr\left( | 1- \lambda_{\text{min}}| \le \epsilon,   | 1- \lambda_{\text{max}}| \le \epsilon  \right).
\end{align*}
The Tracy-Widom law de{s}cribes the marginal distributions of $\lambda_{\text{min}}$ and $\lambda_{\text{max}}$. We would like to avoid working with the joint distribution of the extreme eigenvalues, and instead restrict attention to the distribution of the maximum. Let $\vect{X}$ denote the random vector $\vect{X} = (|1-\lambda_{\text{min}}|, |1-\lambda_{\text{max}}|)^{\T}$. Figure \ref{fig:convergnce} presents some diagrams that will be useful. We wish to know the probability that $\vect{X}$ lies in the  shaded region $C$ in panel (a). For every $\epsilon > 0 $ we have that $\Pr\left( | 1- \lambda_{\text{min}}| \le \epsilon,   | 1- \lambda_{\text{max}}| \le \epsilon  \right) = \Pr(\vect{X} \in C)$. The region $C$ can be expressed as $C=M-R$ where $M$ and $R$ are the shaded regions in panels (b) and (c) respectively. The probability $\Pr (\vect{X} \in M)$ represents the marginal probability that $|1-\lambda_{\text{max}}| \le \epsilon$. The probability $\Pr(\vect{X} \in R)$ represents the probability of the joint event that $(|1-\lambda_{\text{max}}| \le \epsilon,  |1-\lambda_{\text{min}}| > \epsilon)$. We have that
\begin{align*}
\Pr (\vect{X} \in C) &= \Pr (\vect{X} \in M) - \Pr(\vect{X} \in R).
\end{align*}
In panel (c) the dot-dash line gives the identity line where $|1-\lambda_{\text{max}}|=|1-\lambda_{\text{min}}|$. From the first part of the proof of Lemma \ref{lem:gaussian_pointwise_limit} we know that as $d,k$ tends to infinity $\vect{X}$ converges in distribution to the constant vector $\vect{X}_{L}=(|1-(1-\sqrt{\alpha})^2|, |1-(1+\sqrt{\alpha})^2|)^{\T}$. As such, asymptotically $|1-\lambda_{\text{max}}| > |1-\lambda_{\text{min}}|$ with probability one. Referring to panel $(c)$, the random vector $\vect{X}_{L}$ takes values in the region below the dot-dash line with probability one. The limiting random vector $\vect{X}_{L}$ thus satisfies $\Pr(\vect{X}_{L} \in R)=0$ and $\Pr(\vect{X}_{L} \in \partial R)=0$.  As $\vect{X} \overset{d}{\to} \vect{X}_{L}$,  Property \eqref{prop:boundary} of the Portmanteau lemma (Lemma \ref{lem:portmanteau}) gives that $\Pr(\vect{X} \in R)\to \Pr(\vect{X}_{L} \in R) =0$. The limiting probability is then 
\begin{align}
 \underset{d,k \to \infty}{\lim} \Pr\left( | 1- \lambda_{\text{min}}| \le \epsilon,   | 1- \lambda_{\text{max}}| \le \epsilon  \right)  &= \underset{d,k \to \infty}{\lim}\Pr (\vect{X} \in C) \nonumber \\
 &= \underset{d,k \to \infty}{\lim}\Pr (\vect{X} \in M) -\underset{d,k \to \infty}{\lim} \Pr(\vect{X} \in R) \nonumber \\
&=  \underset{d,k \to \infty}{\lim}\Pr (\vect{X} \in M)  - 0 \nonumber \\
&=  \underset{d,k \to \infty}{\lim} \Pr \left( | 1- \lambda_{\text{max}}|  \le \epsilon \right). \label{eq:lambda_max_control}
\end{align}
\begin{figure}
    \centering
    \includegraphics[width=\textwidth]{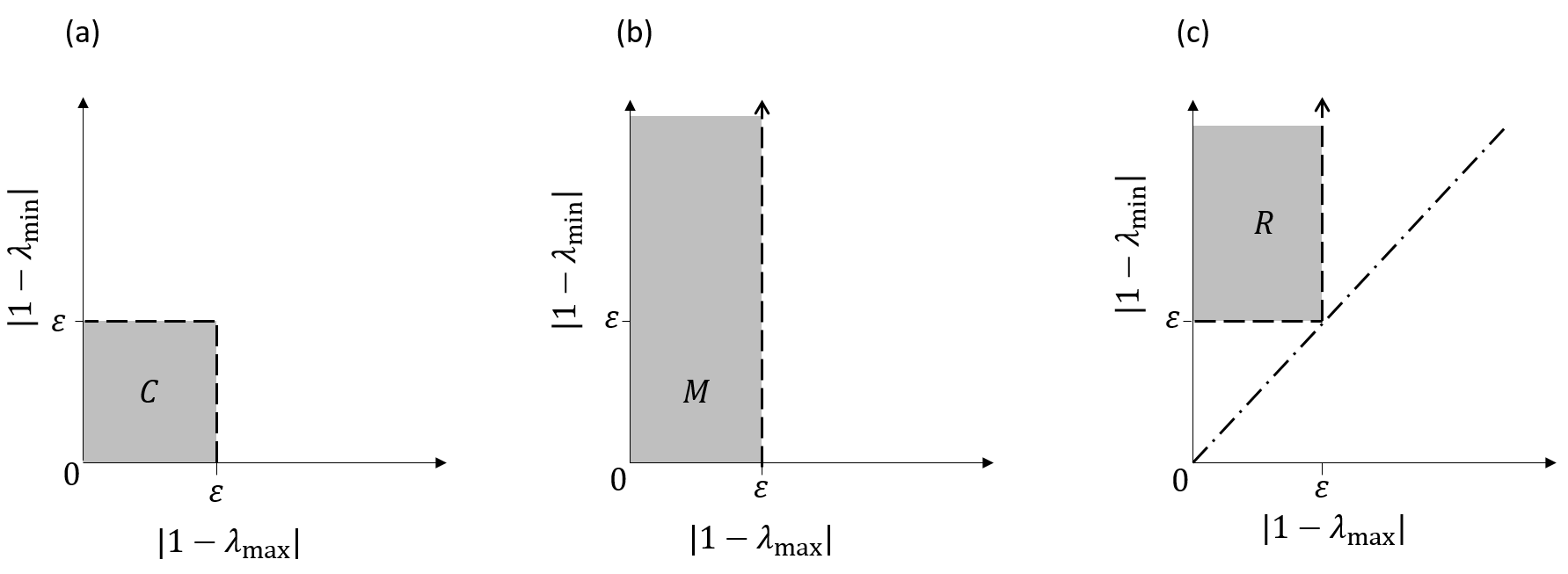}
    \caption[Regions of interest in determining the embedding probability]{Regions of interest in determining the embedding probability. To obtain an $\epsilon$-subspace embedding we require that $|1-\lambda_{\text{min}}|\le \epsilon$ and $|1-\lambda_{\text{max}}| \le \epsilon$. If we define $\vect{X} = (|1-\lambda_{\text{min}}|, |1-\lambda_{\text{max}}|)^{\T}$, we have that $\Pr (\vect{X} \in C) = \Pr (\vect{X} \in M) - \Pr(\vect{X} \in R)$. In panel (c) the dot-dash line gives the identity line where $|1-\lambda_{\text{max}}|=|1-\lambda_{\text{min}}|$.}
    \label{fig:convergnce}
\end{figure}
We have now isolated the maximum eigenvalue $\lambda_{\text{max}}$ as the determining factor in obtaining an $\epsilon$-subspace embedding. We make another application of the Portmanteau lemma to arrive at the final result. From here we can write 
\begin{align}
  \Pr \left( | 1- \lambda_{\text{max}}|   \le \epsilon \right) &= \Pr(\lambda_{\text{max}} \le \epsilon+1)-\Pr(\lambda_{\text{max}} \le 1-\epsilon). \label{eq:port_2}
\end{align}
From Theorem \ref{thm:pointwise_wishart} we know that $\lambda_{\text{max}}$ converges in distribution to the constant random variable $Z_{L}=(1+\sqrt{\alpha})^2$, where we have assumed $\alpha \in (0, 1]$. Let $B$ denote the interval $(-\infty, 1]$. The limiting random variable $Z_{L}$ satisfies $\Pr(Z_{L} \in B) =0$ and $\Pr(Z_{L} \in  \partial B) =0$. As such using property $\ref{prop:boundary}$ of the Portmanteau lemma,  $\underset{d,k \to \infty}{\lim} \Pr(\lambda_{\text{max}} \in B) = 0$. Now $\Pr(\lambda_{\text{max}} \le 1-\epsilon) \le \Pr(\lambda_{\text{max}} \in B)$ for any $\epsilon > 0$.    We can then conclude that $\underset{d,k \to \infty}{\lim}\Pr(\lambda_{\text{max}} \le 1-\epsilon)=0$ for any $\epsilon > 0$. Asymptotically, the term $\Pr(\lambda_{\text{max}} \le 1-\epsilon)$ drops out of the expression for the embedding probability. Taking limits over \eqref{eq:port_2},
\begin{align*}
 \underset{d,k \to \infty}{\lim} \Pr \left( | 1- \lambda_{\text{max}}|   \le \epsilon \right) &= \underset{d,k \to \infty}{\lim} \Pr(\lambda_{\text{max}} \le \epsilon+1)-\underset{d,k \to \infty}{\lim} \Pr(\lambda_{\text{max}} \le 1-\epsilon) \\
 &=  \underset{d,k \to \infty}{\lim} \Pr(\lambda_{\text{max}} \le \epsilon+1) - 0.
\end{align*}
The asymptotic embedding probability is then related to the asymptotic distribution of $\lambda_{\text{max}}$. The inequality can be manipulated to include the centering and scaling constants that appear in Theorem \ref{thm:tw_maximum_eigenvalue}, 
\begin{align*}
\underset{d,k \to \infty}{\lim} \Pr(\lambda_{\text{max}} \le \epsilon+1) &= \underset{d,k \to \infty}{\lim} \Pr\left(\dfrac{\lambda_{\text{max}}-\mu_{k,d}}{\sigma_{k,d}} \le \dfrac{\epsilon+1-\mu_{k,d}}{\sigma_{k,d}} \right).
\end{align*}
Let $Z$ be a random variable with Tracy-Widom distribution $F_{1}$. Now using Theorem \ref{thm:tw_maximum_eigenvalue} in the main text we have that for any fixed $d$ and $k$, it must hold that for any fixed $\epsilon > 0$,
\begin{align}
    \left\lvert \Pr\left(\dfrac{\lambda_{\text{max}}-\mu_{k,d}}{\sigma_{k,d}} \le \dfrac{\epsilon+1-\mu_{k,d}}{\sigma_{k,d}} \right) - \Pr \left( Z \le \dfrac{\epsilon+1-\mu_{k,d}}{\sigma_{k,d}}\right) \right\rvert &\le \label{eq:embedding_sup_error} \\ \underset{{z} \in \mathbb{R}}{\sup} \left\lvert \Pr\left(\dfrac{\lambda_{\text{max}}-\mu_{k,d}}{\sigma_{k,d}} \le z \right) - \Pr(Z \le z) \right\rvert.  \nonumber
\end{align}
As $(\lambda_{\text{max}}-\mu_{k,d})/\sigma_{k,d}$ converges in distribution to the continuous random variable $Z$, where $Z\sim F_{1}$, it follows from Lemma \ref{lem:uniform} that
\begin{align*}
    \underset{d,k \to \infty}{\lim} \   \underset{{z} \in \mathbb{R}}{\sup} \left\lvert \Pr\left(\dfrac{\lambda_{\text{max}}-\mu_{k,d}}{\sigma_{k,d}} \le z \right) - \Pr(Z \le z) \right\rvert =0.
\end{align*}
Now by the squeeze theorem, it holds that for all $\epsilon > 0$,
\begin{align*}
  \underset{d,k \to \infty}{\lim}  \   \left\lvert \Pr\left(\dfrac{\lambda_{\text{max}}-\mu_{k,d}}{\sigma_{k,d}} \le \dfrac{\epsilon+1-\mu_{k,d}}{\sigma_{k,d}} \right) - \Pr \left( Z \le \dfrac{\epsilon+1-\mu_{k,d}}{\sigma_{k,d}}\right) \right\rvert &= 0.
\end{align*}
From Theorem 1 of \cite{ma_accuracy_2012}, the error in the approximation \eqref{eq:embedding_sup_error} is $O(d^{-2/3})$ for even $d$. As discussed in \cite{ma_accuracy_2012}, it is difficult to give a rigorous error bound for odd $d$, however simulations suggest the $O(d^{-2/3})$  bound still holds. 
\end{proof}
\subsection{Proof of Theorem 2}
\begin{proof}Let $\mat{W}\sim \text{Wishart}(k, \mat{I}_{d}/k)$ and let $\lambda_{\text{min}}$ denote the minimum eigenvalue of $\mat{W}$. The probability of convergence can be expressed as
\begin{align}
   \Pr\left(\underset{t \to \infty}{\lim}\lVert\vect{\beta}_{F} - \vect{\beta}^{(t)} \rVert_{2} = 0\right)   &= \Pr\left(\lambda_{\text{min}} > 0.5\right) \nonumber \\
    &= \Pr\left(\dfrac{\log \lambda_{\text{min}}-\nu_{k,d}}{\tau_{k,d}} > \dfrac{\log(0.5)  - \nu_{k,d}}{\tau_{k,d}} \right) \label{eq:convergence_identity}
\end{align}
Let $Z$ be a random variable with Tracy-Widom distribution $F_{1}$. Now from Theorem \ref{thm:tw_minimum_eigenvalue}, $(\log \lambda_{\text{min}} - \nu_{k,d})/\tau_{k,d}$ converges in distribution to the continuous random variable $-Z$, where $Z$ is distributed according to the Tracy-Widom distribution $F_{1}$. For any fixed $d$ and $k$, it must hold that for any fixed $\epsilon > 0$,
\begin{align}
  \left\lvert \Pr\left(\dfrac{\log \lambda_{\text{min}}-\nu_{k,d}}{\tau_{k,d}} > \dfrac{\log(0.5)  - \nu_{k,d}}{\tau_{k,d}} \right) - \Pr \left( -Z > \dfrac{\log(0.5)  - \nu_{k,d}}{\tau_{k,d}}\right) \right\rvert &\le \\
  \underset{{z} \in \mathbb{R}}{\sup} \left\lvert \Pr\left(\dfrac{\log \lambda_{\text{min}}-\nu_{k,d}}{\tau_{k,d}} > z \right) - \Pr(-Z > z) \right\rvert. \label{eq:convergence_sup_error}
\end{align}
 From Lemma \ref{lem:uniform} it must hold that  
\begin{align*}
     \underset{d,k \to \infty}{\lim} \   \underset{{z} \in \mathbb{R}}{\sup} \left\lvert  \Pr\left(\dfrac{\log \lambda_{\text{min}}-\nu_{k,d}}{\tau_{k,d}} > \dfrac{\log(0.5)  - \nu_{k,d}}{\tau_{k,d}} \right)- \Pr(-Z >  z) \right\rvert =0.
\end{align*}
Now by the squeeze theorem, for all $\epsilon > 0$,
\begin{align*}
     \underset{d,k \to \infty}{\lim} \   \left\lvert  \Pr\left(\dfrac{\log \lambda_{\text{min}}-\nu_{k,d}}{\tau_{k,d}} > \dfrac{\log(0.5)  - \nu_{k,d}}{\tau_{k,d}} \right)- \Pr(-Z >  \dfrac{\log(0.5)  - \nu_{k,d}}{\tau_{k,d}} ) \right\rvert =0.
\end{align*}
Rearranging 
\begin{align*}
    \Pr\left(-Z >  \dfrac{\log(0.5)  - \nu_{k,d}}{\tau_{k,d}} \right) &=  \Pr\left( Z \le  \dfrac{\nu_{k,d}-\log(0.5)}{\tau_{k,d}} \right),
\end{align*}
and using the identity \eqref{eq:convergence_identity} gives the the final result,
\begin{align*}
   \quad \underset{n,d,k \to \infty}{\lim} \left\lvert \Pr \left(\underset{t \to \infty}{\lim}\lVert\vect{\beta}_{F} - \vect{\beta}^{(t)} \rVert_{2} = 0\right)- \Pr\left(Z \le \dfrac{\nu_{k,d}-\log (1/2)}{\tau_{k,d}}\right)\right\rvert = 0.
\end{align*}
\end{proof}
From Theorem 2 of \cite{ma_accuracy_2012}, the error in the approximation \eqref{eq:embedding_sup_error} is $O(d^{-2/3})$ for even $d$. As discussed in \cite{ma_accuracy_2012}, it is difficult to give a rigorous error bound for odd $d$, however simulations suggest the $O(d^{-2/3})$  bound still holds. 
\subsection{Proof of Theorem 3}
\begin{proof}
Assumption 1 on the leverage scores is sufficient to establish a central limit theorem for the data-oblivious sketches. 

\begin{theorem}[\cite{ahfock_statistical_2020}]
	\label{thm:sketching_clt}
	Consider a sequence of arbitrary $n \times d$ data matrices $\mat{A}_{(n)}$, where $d$ is fixed. Let \newline $\mat{A}_{(n)}=\mat{U}_{(n)}\mat{D}_{(n)}\mat{V}_{(n)}^{\T}$ represent the singular value decomposition of $\mat{A}_{(n)}$. Let $\mat{S}$ be a $k \times n$ Hadamard or Clarkson-Woodruff sketching matrix where $k$ is also fixed. Suppose that Assumption 1 on the maximum leverage score is satisfied. Then as $n$ tends to infinity
	\begin{align*}
[\widetilde{\mat{A}}\mat{V}_{(n)}\mat{D}_{(n)}^{-1} \mid  \mat{A}_{(n)}]\overset{d}{\to} \emph{MN}(\mat{0}, \mat{I}_{k}, \mat{I}_{d}/k).
	\end{align*}
\end{theorem}

As we only need to consider the sequence of orthonormal matrices $\mat{U}_{(n)}$ to determine the embedding probability, we can use Theorem \ref{thm:sketching_clt} with $\mat{D}_{(n)}$ and $\mat{V}_{(n)}$ set to the $d \times d$ identity matrix. As such we conclude that $\mat{S}_{(n)}\mat{U}_{(n)} \overset{d}{\to} \text{MN}(\mat{I}_{k}, \mat{I}_{d}/k)$. By the continuous mapping theorem it holds that for fixed $d$ and $k$, asymptotically with $n$,  $\mat{U}_{(n)}^{\T}\mat{S}_{(n)}^{\T}\mat{S}_{(n)}\mat{U}_{(n)} \overset{d}{\to} \text{Wishart}(k, \mat{I}_{d}/k)$. Another application of the continuous mapping theorem gives
\begin{align*}
\sigma_{\text{max}}(\mat{I}_{d}-\mat{U}_{(n)}^{\T}\mat{S}^{\T}_{(n)}\mat{S}_{(n)}\mat{U}_{(n)}) \overset{d}{\to} \sigma_{\text{max}}(\mat{I}_{d}-\mat{W}), 
\end{align*}
where $\mat{W} \sim \text{Wishart}(k, \mat{I}_{d}/k)$. We can use the continuous mapping theorem as the limiting Wishart matrix $\mat{W}$ has rank $d$ with probability one. The maximum singular value function is continuous over the range where $\mat{W}$ has full rank \citep{bhatia_matrix_1996}. By the Portmanteau  lemma it then holds that 
\begin{align*}
\underset{n \to \infty}{\lim} \Pr\left( \sigma_{\text{max}}(\mat{I}_{d}-\mat{U}_{(n)}^{\T}\mat{S}_{(n)}^{\T}\mat{S}_{(n)}\mat{U}_{(n)}) \le \epsilon\right)&= \Pr\left(\sigma_{\text{max}}(\mat{I}_{d}-\mat{W}) \le \epsilon \right).
\end{align*}
Now as
\begin{align*}
 \underset{n \to \infty}{\lim} \Pr\left( \mat{S}_{(n)} \text{ is an $\epsilon$-subspace embedding for $\mat{A}_{(n)}$} \right) &=  \underset{n \to \infty}{\lim} \Pr\left( \sigma_{\text{max}}(\mat{I}_{d}-\mat{U}_{(n)}^{\T}\mat{S}^{\T}_{(n)}\mat{S}_{(n)}\mat{U}_{(n)}) \le \epsilon\right) \\
 &= \Pr\left(\sigma_{\text{max}}(\mat{I}_{d}-\mat{W}) \le \epsilon \right),
\end{align*} 
we have the final result.
\end{proof}

\end{document}